\documentclass[11pt,reqno]{article}
\usepackage[utf8]{inputenc} % PDFLaTeX\UTF{C6A9}%
\usepackage[T1]{fontenc}
\usepackage{amssymb, amsmath, amsthm, amsfonts, amscd, epsfig, subcaption}
\usepackage[dvipsnames]{xcolor}
\usepackage[utf8]{inputenc}
\usepackage[english]{babel}
\usepackage{bm}
\usepackage{cleveref}
\usepackage{bbm}
\usepackage{graphicx, epsfig}
\usepackage{geometry,graphicx,pgfplots}
\usepackage{setspace}
\usepackage[export]{adjustbox}
\usepackage[titletoc,toc,title]{appendix}
\usepackage{arydshln}
\usepackage{algorithm}
\usepackage{mathtools}
\usepackage{afterpage}
\usepackage{comment}
\usepackage{cases}

\newtheorem{theorem}{Theorem}[section]

\newtheorem{lemma}[theorem]{Lemma}                                                                                                                                                                                                                                                                             

\newtheorem{definition}{Definition}
\newtheorem{notation}{Notation}[section]

\newtheorem{example}{Example}

\theoremstyle{definition}

\def\RR{\mathbb{R}}
\def\NN{\mathbb{N}}
\def\CC{\mathbb{C}}
\def\ZZ{\mathbb{Z}}

\def\L{\mathcal{L}}
\def\C{\mathcal{C}}
\def\bu{\mathbf{u}}

\def\:={\coloneqq}

\def\G{\mathbf{\Gamma}}

\def\Ss{{\mathcal{S}}}

\def\vp{{\varphi}}
\def\bvp{{\bm{\varphi}}}
\def\bp{{\bm{\phi}}}

\def\Scal{{\mathcal{S}}}

\def\bpsi{\bm{\psi}}

\def\H{\mathbf{H}}
\def\p{\partial}
\def\n{\nabla}

\def\a{\alpha}
\def\b{\beta}

\def\<{\left\langle}
\def\>{\right\rangle}
\def\ta{{\tilde{\alpha}}}
\def\tb{{\tilde{\beta}}}

\def\rmi {\mathrm{i}}

\def\tlambda{\tilde{\lambda}}
\def\tmu{\tilde{\mu}}

\newcommand{\Om}{\Omega}

\newcommand{\eqnref}[1]{(\ref {#1})}
\renewcommand{\qed}{\hfill $\Box$ \medskip}

\def\beq{\begin{equation}}
\def\eeq{\end{equation}}

\newcommand{\ds}{\displaystyle}

\newcommand{\bmat}[1]{\begin{bmatrix}
#1
\end{bmatrix}}

\numberwithin{equation}{section}
\numberwithin{figure}{section}

\definecolor{DH}{HTML}{2ea05e}
\definecolor{DS}{HTML}{408ad8}
\definecolor{MK}{HTML}{cc5050}

\begin{document}
\title{
%Geometric matrix formulation for the plane elastostatic problem and its applications\\
A matrix formulation of the plane elastostatic inclusion problem via geometric function theory \thanks{\footnotesize
This study was supported by the National Research Foundation of Korea (NRF) grants funded by the Korean government (MSIT) (NRF-2021R1A2C1011804, RS-2024-00359109).}}

\author{
Daehee Cho\thanks{Department of Mathematical Sciences, Korea Advanced Institute of Science and Technology, 291 Daehak-ro, Yuseong-gu, Daejeon 34141, Republic of Korea (daehee.cho@kaist.ac.kr, mklim@kaist.ac.kr).}\and 
Doosung Choi\thanks{Department of Mathematics, Applied Mathematics, and Statistics, Case Western Reserve University, Cleveland, OH 44106, USA (doosung.choi@case.edu).} \and  
Mikyoung Lim\footnotemark[2]
}

\date{\today}
\maketitle

\begin{abstract}

We investigate the two-dimensional elastostatic inclusion problem in an unbounded medium. Building on the recent developments for rigid inclusions \cite{Mattei:2021:EAS} and conductivity inclusions \cite{Jung:2021:SEL}, we extend these methodologies to the more general case of elastic inclusions with arbitrary Lam\'{e} constants. Our approach integrates layer potential techniques, geometric function theory, and the complex-variable formulation in plane elasticity.
As a main result, we derive a matrix formulation of the elastostatic inclusion problem using basis functions defined via the exterior conformal mapping of the inclusion. This leads to a series solution framework that incorporates the geometry of the inclusion.
%{\color{red}[comment: need to check] This formulation then yields a series solution for the inclusion problem of general shape subjected to an arbitrary far-field loading. }

\end{abstract}

%
%\noindent {\footnotesize {\bf Mathematics Subject Classification.} {30C35; 35J05; 45P05}}
%

%\tableofcontents

%\vskip 1cm
%{\color{magenta}[comment: need to mention on the convergence of the coefficients such that interchanging the summation works]}

\section{Introduction}

We investigate the elastic inclusion problem in an unbounded medium containing an inhomogeneity whose material properties differ from those of the surrounding medium, under prescribed background solutions. Our primary objective is to determine the resulting perturbed elastic fields. 
Understanding how the shape and material parameters of inhomogeneities affect these fields is essential for analyzing the properties of composite materials \cite{Ammari:2006:EPE,Buryachenko:2007:MHM,Cherkaev:2022:GSE,Dvorak:2013:MCM,Qu:2006:FMS}.
Such inclusion problems, not limited to the elastic equation, have a long history, tracing back at least to the work of Poisson \cite{Poisson:1826:SMS} on the Newtonian potential problem, and later by Maxwell \cite{Maxwell:1873:TEM} on the electromagnetic fields (see also the review papers \cite{Parnell:2016:EHM,Zhou:2013:ARR}).  A noteworthy example is the {\it Eshelby conjecture}, which characterizes the shape of an inclusion that maintain the uniform strain field inside the inclusion under uniform external loadings \cite{Eshelby:1957:DEF,Eshelby:1961:EII,Kang:2008:SPS,Liu:2008:SEC,Sendeckyj:1970:EIP,Ru:1996:EIA,Kang:2009:CPS}.

Inclusions with simple geometries such as disks, spheres, ellipses, and ellipsoids have been thoroughly investigated under uniform loading conditions. These studies have provided explicit solutions using Cartesian or ellipsoidal coordinate systems \cite{Goodier:1933:CSA,Southwell:1926:OCS,Edwards:1951:SCA,Robinson:1951:EEE,Sadowsky:1947:SCA,Sadowsky:1949:SCA}. However, extending these solutions to more complex shapes is not straightforward, primarily due to the interface conditions on the inclusion boundary.
It is worth noting that numerical methods have been proposed to address the inclusion problems with generally shaped inclusions. \cite{Chiu:1977:OTS,Nozaki:1997:EFP,Rodin:1996:EIP,Zou:2010:EPN}.

This paper is focused on the elastostatic inclusion problem in two dimensions to determine the elastic fields in an unbounded isotropic background medium, containing a general inclusion of arbitrary shape under prescribed far-field loadings.

We employ complex analysis techniques, which have proven to be highly effective in addressing two-dimensional elasticity and conductivity inclusion problems.
In particular, the notable work of Muskhelishvili \cite{Muskhelishvili:1953:SBP} laid the foundation for the complex-variable formulation for the elastostatic problem. This formulation reduces the system of elasticity equations to a more tractable scalar-valued complex equation, enabling a simplified yet rigorous analysis, and it has been successfully applied to elastic problems involving inclusions of various shapes \cite{Ammari:2004:RSI,Ando:2018:SPN,Movchan:1997:PSM}, and to the derivation of explicit solutions under uniform far-field loadings \cite{Movchan:1997:PSM,Ru:1999:ASE,Zou:2010:EPN}.

Building on this complex-variable formulation for the elastostatic problem, Mattei and Lim \cite{Mattei:2021:EAS} recently developed an analytic series solution method for the inclusion problem involving a rigid inclusion of general shape under arbitrary far-field loading. This approach also incorporates the layer potential technique and geometric function theory. We construct the basis functions for the series solution using the exterior conformal mapping associated with the inclusion, denoted by $\Psi(w)$. We derive the explicit relations for the expansion coefficients as an infinite matrix equation by applying the rigid boundary condition.

Jung and Lim established a matrix formulation based on geometric function theory and layer potential in the scalar problem for the conductivity (or antiplane elasticity) equation \cite{Jung:2021:SEL}. This framework has been effectively applied to various problems, including shape recovery of planar inclusions \cite{Choi:2021:ASR, Choi:2023:IPP}, the design of neutral inclusions \cite{Choi:2023:GME,Choi:2024:CIV}, and spectral analysis of the Neumann--Poincar{\'e} operator \cite{Jung:2020:DEE,Cherkaev:2022:GSE,Choi:2025:SAN}. 

However, a corresponding matrix formulation for the full elastostatic case with general Lam\'{e} constants has not yet been developed. This is mainly due to the greater analytical complexity of the elasticity system, in contrast to the scalar nature of the conductivity equation. In particular, a key challenge in extending the results of \cite{Mattei:2021:EAS} to the general elastostatic inclusion problem lies in the difficulty of deriving explicit series expansions in powers of $w^{\pm n}$ for terms involving $1/\Psi'(w)$. These terms arise naturally in the layer potential formulation when the exterior conformal mapping is used to define new coordinates, significantly complicating the analysis.

In this work, we address such challenges by establishing cancellation relations for the terms involving $1/\Psi'(w)$. This key observation enables the formulation of an infinite-dimensional linear system that characterizes the elastostatic inclusion problem for arbitrary Lam\'{e} constants. Specifically, we show that the coefficient row vector \(\bm{x}\), representing the density functions in the layer potential formulation of the solution (expanded in a basis associated with the exterior conformal mapping $\Psi(w)$), satisfies an explicit $(8 \times 8)$ block-structured linear system (see Theorem \ref{thm:main:modi} in Section~\ref{subsec:main_results}):
\[
\bm{x}\mathbb{E}\,= -2\, \bm{h},
\]
where \(\mathbb{E}\) is a structured matrix determined by the geometry of the inclusion through the exterior conformal mapping $\Psi$, and $\bm{h}$ is the coefficient row vector of the applied loading in powers of $w^{\pm n}$, assumed to be known. 
%The superscript $T$ denotes transpose. 
The explicit form of \(\mathbb{E}\) is presented later in the paper. The matrix formulation leads to a series solution framework that incorporates the geometry of the inclusion.

The remainder of this paper is organized as follows. Section 2 presents the layer potential technique and complex-variable formulation for the plane elastostatic inclusion problem. Section 3 is devoted to geometric function theory. In Sections 4 and 5, we derive geometric series expansions for the background loading and the single-layer potential. The main result is stated in Section 6, and concluding remarks are provided in Section 7.

\section{Problem formulation}
 Let $\Om$ be an elastic inclusion embedded in $\RR^2$, where $\Om$ (resp., $\RR^2\setminus\overline{\Om}$) is occupied by homogeneous isotropic medium with the bulk modulus $\tilde{\lambda}$ and the shear modulus $\tilde{\mu}$ (resp. $\lambda$ and $\mu$). 
We assume that $\Om$ is bounded, simply connected with an analytic boundary. 

In the absence of an inclusion, a background displacement field $\bu_0$ in a homogeneous linear elastic medium with Lam\'{e} constants $\lambda$ and $\mu$, and subject to no body forces, satisfies
$$
 \nabla\cdot\bm\CC_0\widehat{\n}\bu_0=0, \quad \CC_0:=(C_{kl}^{ij})_{i,j,k,l=1}^2, \quad C_{kl}^{ij}=\lambda\delta_{ij}\delta_{kl}+\mu(\delta_{ik}\delta_{jl}+\delta_{il}\delta_{jk}),
$$
where $\widehat{\n}\bu_0 =\frac{1}{2}(\n\bu_0+(\n\bu_0)^T)$ is the symmetric strain tensor, $T$ denotes the matrix transpose, and $\delta_{ij}$ is the Kronecker delta. 
It follows that $\L_{\lambda,\mu}\bu_0=0$ in $\RR^2$, where $\L_{\lambda,\mu}$ is  the differential operator given by
\begin{align*}
\L_{\lambda,\mu}\bu_0=\mu\Delta\bu_0+(\lambda+\mu)\nabla\nabla\cdot\bu_0\quad\mbox{in }\RR^2.
\end{align*}
We assume $\mu>0$, $\lambda+\mu>0$ so that $\L_{\lambda,\mu}$ is elliptic.  The conormal derivative (or the traction term) is defined by
\begin{align}\notag%\label{def:conormal}
\frac{\partial\bu}{\partial\nu}=\left(\mathbb{C}_0 \widehat{\n}\bu\right)N=\lambda(\nabla\cdot\bu)N+\mu(\nabla\bu+\nabla \bu^T)N\quad\mbox{on }\p \Om,
\end{align}
where $N$ is the outward unit normal vector to $\p \Om$. 

Analogously, for Lam\'{e} constants $(\tilde{\lambda},\tilde{\mu})$, we define $\mathbb{C}_1$, $\mathcal{L}_{\tilde{\lambda},\tilde{\mu}}$, and $\partial / \partial \tilde{\nu}$. We assume $\tilde{\mu}>0, \tilde{\lambda}+\tilde{\mu}>0$, and $(\lambda -\tilde{\lambda})^2+(\mu-\tilde{\mu})^2\neq 0$.

In the presence of the inclusion $\Om\subset\RR^2$, the displacement field $\bu$ satisfies the transmission problem:
\begin{align}\label{eqn:main:trans}
\begin{cases}
\ds\mathcal{L}_{\lambda,\mu} \mathbf{u} = 0&\text{in }\RR^2\setminus\overline{\Om},\\
\ds \mathcal{L}_{\tilde{\lambda},\tilde{\mu}} \mathbf{u} = 0 &\text{in }\Omega,\\
\ds\mathbf{u}\big|^+ = \mathbf{u}\big|^-  &\text{on } \partial \Om,\\
\ds\dfrac{\partial \mathbf{u}}{\partial \nu}\Big|^+ = \dfrac{\partial \mathbf {u}}{\partial \tilde{\nu}}\Big|^- &\text{on } \partial \Om,\\
\ds(\mathbf{u-H})(x) = O(|x|^{-1}) &\text{as }x\to +\infty,
\end{cases}
\end{align}
where $\mathbf{H}$ is a far-field loading satisfying $\L_{\lambda,\mu}\mathbf{H}=0$ in $\RR^2$. We assume each component of $\mathbf{H}$ is a real-valued polynomial.

We focus on the analytic solution to \eqnref{eqn:main:trans} for a given arbitrary $\Om$ and far-field loading $\mathbf{H}$.
As a main result, we derive a matrix formulation for the planar elastostatic inclusion problem using a basis defined via the exterior conformal mapping of the inclusion (see Theorem \ref{thm:main:modi}).

\subsection{Layer potential formulation}\label{subsec:lame}

Let $\G=(\Gamma_{ij})_{i,j=1}^2$ be the Kelvin matrix of the fundamental solution to the Lam\'e operator $\L_{\lambda,\mu}$. For $x\in\RR^2\setminus\{0\}$,  the elements of the matrix are defined by
\begin{gather}\label{def:alpha:beta}
\Gamma_{ij}({x})
=\ds\frac{\a}{2\pi}\delta_{ij}\log|x|-\frac{\b}{2\pi}\frac{x_ix_j}{|x|^2}
\quad
\mbox{for }
\a=\tfrac{1}{2}(\tfrac{1}{\mu}+\tfrac{1}{2\mu+\lambda}), \ \b=\tfrac{1}{2}(\tfrac{1}{\mu}-\tfrac{1}{2\mu+\lambda}).
\end{gather}
For $\bvp=(\varphi_1,\varphi_2)\in L^2(\p \Om)^2$, we define the single-layer potential associated with $\L_{\lambda,\mu}$ by
 \begin{align*}
\Ss_{\p \Om}[\bvp](x)=&\int_{\partial \Om}\G(x-y)\bvp(y)d\sigma(y),\quad x\in\RR^2.
\end{align*}
It holds that $\L_{\lambda,\mu}\Ss_{\p \Om}[\bvp](x)=0$ in $\Om\cup(\RR^2\setminus\overline{\Om})$. The conormal derivative of this potential admits the jump relation:
\begin{align}\label{jump:single}
\dfrac{\partial}{\partial \nu}\Ss_{\p \Om}[\bvp]\Big|^\pm = \Big(\pm\frac{1}{2}I + \mathcal{K}_{\partial \Om}^*\Big)[\bvp]\quad\mbox{on }\partial \Om,
\end{align}
where $+$ and $-$ respectively indicate the limits from outside and inside $\Om$, and  $I$ is the identity operator on $L^2(\p\Om)^2$, and 
\begin{align}
\mathcal{K}_{\partial \Om}^*[\bvp](x) = \mbox{p.v.}\int_{\partial\Om}\dfrac{\partial}{\partial \nu_x}\G(x-y)\bvp(y)\, d\sigma(y),\quad x\in\partial \Om.
\end{align}
Here, p.v. denotes Cauchy's principal value. 

Analogously, for Lam\'{e} constants $(\tilde{\lambda},\tilde{\mu})$, we define the corresponding single-layer potential $\tilde{\mathcal{S}}_{\partial \Omega}$ and its conormal derivative $\frac{\partial}{\partial_{\tilde{\nu}}}\tilde{\mathcal{S}}_{\partial \Omega}$. We also define $\ta$ and $\tb$ in the same manner as in \eqref{def:alpha:beta}, using $(\tilde{\lambda}, \tilde{\mu})$.

We denote by $\Psi$ the space of rigid displacements, i.e.,
$
\mathcal{R}=\mbox{span}\left\{(1,0),(0,1),(x_2,-x_1)\right\},$
and set $$
L^2_\mathcal{R}(\partial \Om)=\big\{\bm{f}\in L^2(\partial \Om)^2:\textstyle\int_{\partial \Om}{\bm{f}}\cdot{\bm \theta} \, d\sigma=0\ \mbox{for all }{\bm \theta} \in \mathcal{R}\big\}.$$
For any $\bpsi\in\mathcal{R}$, it holds that $\L_{\lambda,\mu}\bpsi=0$ and $\frac{\p}{\p \nu} \bpsi=0$ on any interface.
If ${\bm{f}}\in L^2_\mathcal{R}(\partial \Om)$, then $\Ss_{\p \Om}[{\bm{f}}](x)=O(|x|^{-1})$ as $|x|\rightarrow +\infty$. 

From the jump relation \eqref{jump:single}, the solution to \eqnref{eqn:main:trans} can be expressed as
\begin{align}\label{sol:formula1}
\mathbf{u}(x) =
\begin{cases}
\mathbf{H}(x) + \mathcal{S}_{\partial \Omega}[\bpsi](x),& x\in\mathbb{R}^2\setminus\overline{\Omega},\\
\tilde{\mathcal{S}}_{\partial \Omega}[\bm{\phi}](x),& x\in \Omega,
\end{cases}
\end{align}
where $(\bm{\phi},\bpsi)\in L^2(\partial \Omega)^{2}\times L^2_{\mathcal{R}}(\partial \Omega )$ is the unique solution to the following boundary integral equations (refer to \cite{Escauriaza:1993:RPS} for the solvability):
\begin{align}\label{sol:formula2}
\begin{cases}
\ds \tilde{\mathcal{S}}_{\partial \Omega }[\bp]\big|^- - \mathcal{S}_{\partial \Omega}[\bpsi]\big|^+ = \mathbf{H} &\mbox{on }\partial \Omega ,\\[1mm]
\ds\dfrac{\partial }{\partial \tilde{\nu}}\tilde{\mathcal{S}}_{\partial \Omega }[\bp]\Big|^- - \dfrac{\partial }{\partial \nu }\mathcal{S}_{\partial \Omega }[\bpsi]\Big|^+ = \dfrac{\partial \mathbf{H}}{\partial \nu}&\mbox{on }\partial \Omega. 
\end{cases}
\end{align}

\subsection{Complex-variable formulation}\label{sec:complex}
Throughout this paper, we identify $\RR^{2}$ with $\CC$ by mapping $x=(x_1,x_2)$ to $z=x_1+\rmi x_2$.

For a vector-valued function $\bu(\cdot,\cdot)$ taking values in $\RR^2$, we adopt the complex notation:
\beq\label{complex:expression}
\begin{aligned}
u(z)&=\big(\bu(z)\big)_1+\rmi \,\big(\bu(z)\big)_2,\\
\partial_{\nu}u(z)& = \big(\dfrac{\partial \bu}{\partial \nu}(z)\big)_1 + \rmi \,\big( \dfrac{\partial \bu}{\partial \nu}(z)\big)_2,
\end{aligned}
\eeq
where $(\cdot)_j$ denotes the $j$-th component of a vector. Similarly, we define $\psi$, $\phi$, and $H$ as the complex-valued counterparts of $\bpsi$, $\bp$, and $\H$ given in \eqref{sol:formula1}, respectively.

\begin{lemma}[\cite{Ammari:2007:PMT,Muskhelishvili:1953:SBP}]\label{Complex_representation_of_solutions}
Let $\bu$ be the solution to \eqnref{eqn:main:trans} and $u$ be its corresponding complex function as in \eqref{complex:expression}.
Set $u_e:=u|_{\CC\setminus \overline{\Om}}$ and $u_i:=u|_{\Om}$. Then there are functions $f_e$ and $g_e$ holomorphic in $\CC\setminus\overline{\Om}$ and $f_i$ and $g_i$ holomorphic in $\Om$ such that
\begin{align}\label{complex:ue}
u_e(z)&={\kappa} f_e(z)-z\,\overline{{f_e}'(z)}-\overline{g_e(z)},\quad z\in\CC\setminus\overline{\Om},\\\label{complex:ui}
u_i(z)&=\tilde{\kappa} f_i(z)-z\,\overline{{f_i}'(z)}-\overline{g_i(z)},\quad z\in\Om,
\end{align}
where 
$$\kappa=\frac{\lambda+3\mu}{\lambda+\mu},\quad \tilde{\kappa}=\frac{\tlambda+3\tmu}{\tlambda+\tmu}.$$
Moreover,  for some constant $c$ and $z\in\p\Om$,
\begin{align} \label{complex:trans:1}
\kappa f_e(z)-z\,\overline{f_e'(z)}-\overline{g_e(z)}
&=\tilde{\kappa} f_i(z)-z\,\overline{f_i'(z)}-\overline{g_i(z)},\\ \label{complex:trans:2}
2\mu\left(f_e(z)+z\,\overline{f_e'(z)}+\overline{g_e(z)}\right)
&=2\tmu\left(f_i(z)+z\,\overline{f_i'(z)}+\overline{g_i(z)}\right)+c.
\end{align}
\end{lemma}

\smallskip
For convenience, we introduce operators
\beq\label{def:I}
\begin{aligned}
\mathcal{I}^e[u_e](z)&:=\mu\left(f_e(z)+z\overline{f_e'(z)}+\overline{g_e(z)}\right),\\
\mathcal{I}^i[u_i](z)&:=\tilde{\mu}\left(f_i(z)+z\overline{f_i'(z)}+\overline{g_i(z)}\right),
\end{aligned}
\eeq
written by notational abuse as the right-hand side of \eqref{def:I} may differ by some constant depending on the choice of $f_e$ and $g_e$ (or $f_i$ and $g_i$).

It is important to note that equation \eqref{complex:trans:2} is derived from the complex-variable representation of the exterior conormal derivative of $u_e$ (along with an analogous interior formula for $u_i$):
\begin{align}\label{Complex_representation_of_solutions_2}
\left(\partial_{\nu}u_e(z)\right) d\sigma(z)
=-2\mu\,\partial\Big[f_e(z)+z\,\overline{f_e'(z)}+\overline{g_e(z)}\,\Big],
\end{align}
where $d\sigma$ is the line element of $\p\Om$, $\partial=\frac{\partial}{\partial x_1}dx_1+\frac{\partial}{\partial x_2}dx_2=\frac{\p}{\p z}dz+\frac{\p}{\p\overline{z}}d\overline{z}$, and $f'(z)=\frac{\p f}{\p z}(z).$
The interface conditions \eqref{complex:trans:1} and \eqref{complex:trans:2} are equivalent to the following conditions:\beq\label{interface:modified}
u_e=u_i\quad\mbox{and}\quad\mathcal{I}^e[u_e]=\mathcal{I}^i[u_i]+\mbox{constant}. 
\eeq

The complex-valued single-layer potentials and their conormal derivatives associated with $\varphi$ and $\phi$ are defined analogously to \eqref{complex:expression}. For instance,
$$
\begin{aligned}
S_{\p\Om}[\psi](z)&=\big(\Scal_{\p\Om}[\bpsi](z)\big)_1+\rmi \,\big(\Scal_{\p\Om}[\bpsi](z)\big)_2,\\
\partial_{\nu}S_{\p\Om}[\psi](z) &= \big(\dfrac{\partial }{\partial \nu}\Scal_{\p\Om}[\bpsi](z)\big)_1 + \rmi \,\big( \dfrac{\partial }{\partial \nu}\Scal_{\p\Om}[\bpsi](z)\big)_2.
\end{aligned}
$$
With these definitions, the representation formula \eqref{sol:formula1} can be written in complex form as
\begin{align}\label{complex:sol:u}
u(z) = 
\begin{cases}
H(z) + S_{\partial \Om}[\psi](z),&\qquad z\in\mathbb{C}\setminus\overline{\Omega},
\\
\widetilde{S}_{\partial \Omega}[\phi](z),&\qquad z\in \Omega.
\end{cases}
\end{align}

\subsection{Holomorphic function expressions for the single-layer potentials}\label{subsec:CIE}

Similarly to \eqref{complex:ue} and \eqref{complex:ui}, the single-layer potential $S_{\p\Om}$ has a decomposition in terms of two holomorphic functions, each of which admits an explicit boundary-integral representation.
More precisely, as shown in \cite[Theorem 9.20]{Ammari:2007:PMT} (see also \cite[Appendix A]{Ando:2018:SPN}), for any density $\varphi\in L^2(\p\Om)^2$, it follows that
\begin{align}\label{asymptotic_general_single_layer}
2 S_{\p \Om}[\vp](z)=
2\a\,L[\vp](z)-\b z\,\overline{\C[\vp](z)}+\b\,\overline{\C[\overline{\zeta}\vp](z)}-c_\varphi \quad \mbox{for } z\in\Om \cup \CC\setminus\overline{\Om}.
\end{align}
Here, $\alpha$ and $\beta$ are given in \eqnref{def:alpha:beta}, 
$$c_\varphi=\frac{\beta}{2\pi}{\int_{\partial \Om}{\varphi(\zeta)}d\sigma(\zeta)}.$$
The integral operators are defined by
\begin{align}\label{def:C}
\begin{aligned}
L[\vp](z)&=\frac{1}{2\pi}\int_{\p \Om}\ln|z-\zeta|\vp(\zeta)\,d\sigma(\zeta),\\
\C[\vp](z)&=\frac{1}{2\pi}\int_{\p \Om}\frac{\vp(\zeta)}{z-\zeta}\,d\sigma(\zeta),\\
\C[\overline{\zeta}\vp](z)&=\frac{1}{2\pi}\int_{\p \Om}\frac{\vp(\zeta)}{z-\zeta}\,\overline{\zeta}\,d\sigma(\zeta).
\end{aligned}
\end{align}
The operator $L$ corresponds to the single-layer potential for the Laplacian. In particular, as shown in \cite[Theorem 5.1]{Jung:2021:SEL},
\beq\label{L:varphi0}
L[\varphi_0](z)=
\begin{cases}
\ds \ln \gamma,\quad &z\in \Om,\\
\ds \ln |w|,\quad &z\in \CC\setminus\overline{\Om},
\end{cases}
\eeq
where $\gamma>0 $ is the conformal radius of $\Om$, and $\varphi_0$ is defined in \eqref{exterior_basis}. 
Note that $L[\varphi_0](z)$ cannot be the real part of a holomorphic function defined throughout $\CC\setminus\overline{\Om}$.

\begin{definition}\label{def:SL:holomorphic}
For  $\varphi\in L^2(\p\Om)^2$ and $z\in \CC\setminus\p\Om$, we define
\begin{align}
\label{def:L}
\begin{aligned}
\L[\vp](z)&:=
\begin{cases}
\ds\frac{1}{2\pi}\int_{\p \Om}\log(z-\zeta)\vp(\zeta)\,d\sigma(\zeta)\quad&\mbox{if }\int_{\p\Om}\vp\,d\sigma=0,\\[3mm]
 \ds c\,\ln \gamma\quad&\mbox{if }\vp\equiv c\,\varphi_0,\ z\in\Om,
\end{cases}
\end{aligned}
\end{align}
where $c$ is an arbitrary constant. 
\end{definition}
If $\int_{\p\Om} \varphi \,d\sigma=0$, then $\L[\vp](z)$ is well-defined, by taking a proper branch cut of $\log(z-\zeta)$, and holomorphic in $\Om\cup(\CC\setminus\overline{\Om})$. If instead $\varphi=c\,\varphi_0$ for some constant $c$, then $\L[\vp](z)$  is constant, and thus trivially holomorphic.
By combining Definition~\ref{def:SL:holomorphic} with \eqref{L:varphi0}, one obtains
 \beq\label{L:calL}
L[\vp]=\frac{1}{2}\big(\L[\vp]+\overline{\L[\overline{\vp}]}\big),
\eeq
which holds in $\Om$ for all $\vp\in L^2(\p\Om)^2$. 
Moreover, if $\int_{\p\Om}\vp d\sigma=0$, this identity extends to $\Om\cup\CC\setminus\overline{\Om}$, and $$\L[\vp]'(z):=\frac{\p}{\p z}\L[\vp](z)=\C[\vp](z),\quad z\in \Om\cup(\CC\setminus\overline{\Om}).$$
In addition, the continuity of the Laplacian single-layer potential $L[\varphi]$ implies the following continuity along the interface: 
\beq\label{prop:C2:main:c}
\big(\L[\vp]+\overline{\L[\overline{\vp}]}\big)\Big|_{\p\Om}^+ = \big(\L[\vp]+\overline{\L[\overline{\vp}]}\big)\Big|_{\p\Om}^-,
\eeq
again assuming $\int_{\p\Om}\vp\, d\sigma=0$.

The relations \eqnref{asymptotic_general_single_layer} and \eqref{L:calL} (together with $\kappa\beta=\alpha$) lead the following lemma. 
\begin{lemma} [\cite{Ammari:2007:PMT,Ando:2018:SPN}]\label{lemma:S:complex:all}
For $\vp\in L^2(\Om)^2$ and $z\in \Om$ (or if $\int_{\p\Om}\vp\,d\sigma=0$ and $z\in \CC\setminus\overline{\Om}$), we have
\beq\label{S:complex:all}
\begin{aligned}
\ds 2 S_{\p \Om}[\varphi] (z)&=\kappa f(z)-z\overline{f'(z)}-\overline{g(z)}-c_\varphi,\\
\ds	f(z)&=\b\,\L[\vp](z),\\
\ds g(z)&=-\a\,\L[\overline{\vp}](z)-\b\,\C[\overline{\zeta}\vp](z)
\end{aligned}
\eeq
with $c_\varphi=\frac{\beta}{2\pi}{\int_{\partial \Om}{\varphi(\zeta)}d\sigma(\zeta)}$, and $f(z)$ and $g(z)$ are holomorphic.
\end{lemma}

Applying Lemma \ref{lemma:S:complex:all} along with the fact that  $\int_{\p\Om}\psi\,d\sigma=0$, we derive the following:\begin{lemma}\label{lemma:Single:layer}
Let $\psi$ and $\phi$ be the density functions from \eqref{complex:sol:u}. Then, for $z\in\p\Om$, we have
\begin{align}\notag
\begin{cases}
\ds
2 S_{\p \Om}[\psi]\Big|^+(z)=
\a\left(\mathcal{L}[\psi](z)+\a\,\overline{\mathcal{L}[\overline{\psi}](z)}\right)+ \beta\left(-z\,\overline{\C[\psi](z)}+\,\overline{\C[\overline{\zeta}\psi](z)}\right),\\[2mm]
\ds \mathcal{I}^e\big[2S_{\partial \Omega}[\psi](z)\big]\Big|^+(z)
= \mu\left(\b\,\mathcal{L}[\psi](z)-\a\,\overline{\mathcal{L}[\overline{\psi}](z)}-\b \left(-z\,\overline{\C[\psi](z)}+\,\overline{\C[\overline{\zeta}\psi](z)}\right)\right)
\end{cases}
\end{align}
and
\begin{align}\notag
\begin{cases}
\ds
2 \widetilde{S}_{\p \Om}[\phi]\Big|^-(z)=
\ta\left(\mathcal{L}[\phi](z)+\overline{\mathcal{L}[\overline{\phi}](z)}\right)+\tb \left(-z\,\overline{\C[\phi](z)}+\,\overline{\C[\overline{\zeta}\phi](z)}\right)-c_\phi,\\[2mm]
\ds \mathcal{I}^i\big[2\widetilde{S}_{\partial \Omega}[\phi]\big]\Big|^-(z)
= \tmu\left(\tb\,\mathcal{L}[\phi](z)-\ta\,\overline{\mathcal{L}[\overline{\phi}](z)}-\tb \left(-z\,\overline{\C[\phi](z)}+\,\overline{\C[\overline{\zeta}\phi](z)}\right)\right),
\end{cases}
\end{align}
where $c_\phi=\frac{\tilde{\beta}}{2\pi}{\int_{\partial \Om}{\phi(\zeta)}d\sigma(\zeta)}$ and the single-layer potential are scaled by a factor of $2$, consistent with the expression in Lemma~\ref{lemma:S:complex:all}. 
The operators $\mathcal{I}^e$ and $\mathcal{I}^i$ are defined in \eqref{def:I} with unspecified additive constants. 
\end{lemma}

\section{Exterior conformal mapping, Faber polynomials and their associated matrices}\label{subsec:pre2}

In this section, we introduce the exterior conformal mapping and its corresponding Faber polynomials of the domain $\Om$. We also define the associated matrices that will be used to express the transmission problem \eqref{sol:formula1} in matrix form.

\subsection{Exterior conformal mapping}
As we assume $\Om$ is simply connected and bounded, by the Riemann mapping theorem, there exists a unique constant $\gamma>0$ and a conformal mapping from $\{w\in\CC:|w|>\gamma\}$ onto $\mathbb{C}\setminus\overline{\Omega}$ with $\Psi(\infty)=\infty$ and $\Psi'(\infty)=1$ (see, for instance, \cite[Chapter 1.2]{Pommerenke:1992:BBC}). 
This mapping admits the Laurent series expansion:
\beq\label{eqn:extmapping} 
\Psi(w)=w+{a}_0+\frac{{a}_1}{w}+\frac{{a}_2}{w^2}+\cdots,\quad |w|>\gamma.%=w+\sum_{k=0}^{\infty}a_kw^{-k}.
\eeq 
For notational convenience, we set $a_{-1}=1$. 
The constant $\gamma$ is called the {\it conformal radius} of $\Om$. We set $\rho_0=\ln \gamma.$
The map $\Psi$ admits a continuous extension to $\p\Om$ by the Carath\'eodory extension theorem \cite{Caratheodory:1913:GBR} and, furthermore, $\Psi'$ has a continuous extension to $\p\Om$ for $\Om$ with a $C^{1,\alpha}$ boundary by the Kellogg--Warschawski theorem \cite{Pommerenke:1992:BBC}.

We assume that $\Om$ has an analytic boundary, meaning $\Psi$ can be conformally extended to $|w|>\gamma-\delta$ for some $\delta>0$. Consequently, the Grunsky coefficients satisfy sharper bounds than those in \eqref{c:bound}, with $\gamma$ replaced by $\gamma-\delta$.

Using the exterior conformal mapping $\Psi$ associated with $\Om$, 
  we introduce modified polar coordinates $(\rho,\theta)\in[\rho_0,\infty)\times[0,2\pi)$ for $z\in\CC\setminus\overline{\Om}$ via the conformal map
 $$z=\Psi(w),\quad w=e^{\rho+\mathrm{i}\theta}.$$
To simplify the notation, we write $$\Psi(\rho,\theta):=\Psi(e^{\rho+\mathrm{i}\theta}).$$
The scaling factor associated with this mapping,
$
h(\rho, \theta) := \left|\frac{\partial \Psi} {\partial \rho}\right| = \left|\frac{\partial \Psi} {\partial \theta}\right|,
$
is the same in both the $\rho$- and $\theta$-directions. The arc-length element on the boundary $\p\Om$ (corresponding to $\rho=\rho_0$) is given by
\beq\label{dsigma}
d\sigma(z)=h(\rho_0,\theta)d\theta \quad \mbox{on } \p\Om.
\eeq

We define the complex-valued basis functions
\begin{equation}\label{exterior_basis}
\varphi_k(z)=\frac{\mathrm{e}^{\mathrm{i}k\theta}}{h(\rho_0,\theta)}\quad\mbox{for }k\in\ZZ,
\end{equation}
where $z=\Psi(\rho_0,\theta)$. The collection $\{\varphi_k\}_{k\in\ZZ}$ forms an orthogonal system in $L^2(\p\Om,d\sigma)$, providing a convenient basis for expanding density functions on $\p\Om$.

\subsection{Faber polynomials}
The exterior conformal mapping $\Psi$ defines the so-called {\it Faber polynomials} $F_m(z)$ via the generating relation \cite{Faber:1903:UPE}:
\begin{equation}\label{eqn:Fabergenerating}
\frac{w\Psi'(w)}{\Psi(w)-z}=\sum_{m=0}^{\infty}F_m(z)w^{-m}\quad \mbox{for }|w|>\gamma,\ z\in\overline{\Om}.
\end{equation}
As a consequence, the complex logarithmic functions satisfy with appropriate branch cuts that
\beq\label{Faber:log:expan}
\log(\Psi(w)-z)=\log w -\sum_{m=1}^\infty \frac{1}{m} F_m(z) w^{-m}\quad \mbox{for }|w|>\gamma,\ z\in\overline{\Om}.
\eeq
Notably, $F_m(\Psi(w))$ has only a single positive order term as 
\begin{equation}\label{eqn:Faberdefinition}
F_m(\Psi(w))%=w^m+\frac{c_{m,1}}{w}+\frac{c_{m,2}}{w^2}+\cdots=
=w^m+\sum_{k=1}^{\infty}c_{mk}{w^{-k}},
\end{equation}
where $c_{mk}$ is called the \textit{Grunsky coefficient}. 
It holds  for all $m,k\geq 1$ that $k c_{mk}=m c_{km}$ \cite{Grunsky:1939:KSA,Duren:1983:UF} and (see, for instance,\cite{Mattei:2021:EAS})
\beq\label{c:bound}
\left|c_{mk}\right|\leq 2 m \gamma^{m+k}.
\eeq

The Faber polynomials can be determined from the coefficients $\{ a_n \}_{0 \leq n \leq m-1}$ via the recursion relation \cite{Grunsky:1939:KSA,Duren:1983:UF}:
\beq\label{Faberrecursion}
F_{m+1} (z) = z F_m (z) - m a_m -\sum_{n=0}^{m}  a_n F_{m-n} (z), \quad m\ge 0.
\eeq
For example, the first few Faber polynomials are
\beq\label{Faber:examples}
\begin{aligned}
F_0(z)&=1,\quad
F_1(z)=z-a_0,\quad
F_2(z)=z^2-2a_0 z+a_0^2-2a_1,\\
F_3(z)& = z^3 - 3a_0z^2 + 3(a_0^2-a_1)z - a_0^3+3a_0a_1-3a_2.
\end{aligned}
\eeq

\subsection{Matrices associated with the conformal mapping and Faber polynomials}\label{subsec:matrices}

 All matrices considered here are semi-infinite, indexed from $0$ to $\infty$, are of the form $[x_{mn}]_{m,n=0}^\infty$. 
We begin by defining the following diagonal matrices:
\beq\label{Mat:D}
\begin{aligned}
{\mathcal{N}}^{\pm 1}&=\text{diag}\,\big(1, \, 1, \, 2^{\pm 1}, \, 3^{\pm 1}, \, \cdots\big),\\
\mathcal{N}^{\pm 1}_0&=\text{diag}\,\big(0, \, 1, \, 2^{\pm 1}, \, 3^{\pm 1}, \, \cdots\big),\\
I_0 &=  \text{diag}\,\big(0, \, 1, \, 1, \, 1,\,\cdots\big) 
\end{aligned}
\eeq
and, for any $s>0$ and $k\in\NN$, 
\beq\label{Mat:D2}
\begin{aligned}
s^{k\mathcal{N}}&=\text{diag}\,\big(1,\, s^{k},\, s^{2k},\,s^{3k},\, \cdots\big),\\
s^{k\mathcal{N}}_0&=\text{diag}\,\big(0,\, s^{k},\, s^{2k},\, s^{3k},\, \cdots\big).
\end{aligned}
\eeq
We also define
\beq\label{def:mat:T}
{T}:=\left[ (n+1)\delta_{m}^{n+1}\right]_{m,n=0}^\infty=
\begin{bmatrix}
0& 0 & \hskip 1mm 0 \hskip 1mm &\hskip 1mm 0  \  \cdots\\[1mm]
1 &0 & 0 & 0   \  \cdots\\[1mm]
0 &2& 0 & 0  \ \cdots\\[1mm]
0&0& 3 & 0 \  \cdots\\[1mm]
\vdots & \vdots & \vdots & \vdots\ \ddots 
\end{bmatrix},
\eeq
In the following, we define matrices determined by the exterior conformal mapping $\Psi$ and the Faber polynomials $F_m$. 

\paragraph{Matrices associated with $\Psi$.}
We denote by $a_m$ the coefficients of the conformal mapping $\Psi$ as in \eqref{eqn:extmapping}, adopting the convention that $$a_{-1}=1,\quad a_{-n}=0\quad\mbox{for all }n\geq 2.$$
With a slight abuse of notation, we use $\Psi_+$, $\Psi_-$, and $\Psi_0$ to denote the semi-infinite matrices indexed by $m,n \geq 0 $, defined as follows:
\beq\label{def:Psi:mat}
\begin{aligned}
&[\Psi_+]_{mn}=a_{m+n},\\
& [\Psi_-]_{mn}=a_{m-n},\\
&[\Psi_0]_{mn}=
\begin{cases}
a_0\quad & \mbox{if }(m,n) = (0,0),\\
a_{-1} \quad & \mbox{if }(m,n)=(1,0), (0,1),\\
0&\mbox{otherwise}.
\end{cases}
\end{aligned}
\eeq
Here, $[\cdot]_{mn}$ denotes the $mn$-th entry of a semi-infinite matrix. 

\smallskip
\smallskip

\paragraph{Grunsky coefficient matrix $C$.}
Recall that $c_{mn}$ are the Grunsky coefficients.
For notational convenience, we set $$c_{0n}=c_{m0}=0\quad\mbox{for all }m,n\geq 0.$$ We define the matrix $C$ by
\beq\label{def:mat:C}
{C}:=\left[c_{mn}\right]_{m,n=0}^\infty=
\begin{bmatrix}
0& 0 & \hskip 1mm 0 \hskip 1mm &\hskip 1mm 0  \  \cdots\\[1mm]
0 &c_{11} & c_{12}& c_{13}  \  \cdots\\[1mm]
0 &c_{21}& c_{22} & c_{23}  \ \cdots\\[1mm]
0&c_{31}& c_{32} & c_{33} \ \cdots\\[1mm]
\vdots & \vdots & \vdots & \vdots\ \ddots 
\end{bmatrix}.
\eeq

\paragraph{Expansion coefficient matrix $P$ for $F_m$.} 
The $m$-th Faber polynomial has the form
\beq\label{FaberP}
F_m(z) =\sum_{n=0}^{m} p_{mn} z^n,\quad m\geq 0,
\eeq
where each coefficient $\{ p_{mn} \}_{0\leq n \leq m}$ depends only on the coefficients $\{ a_n \}_{0 \leq n \leq m-1}$. From the recurrence relation \eqref{Faberrecursion}, it follows that $$p_{mm} = 1,\quad p_{(m+1)m} = - (m+1) a_0\mbox{ for }m\ge0,$$ and 
\begin{align*}
p_{(m+1)0} &= -ma_m - { \sum_{k=0}^m a_{m-k}}\, p_{k0}\quad\mbox{for }m\ge0,\\
p_{(m+1)n} &= p_{m(n-1)} - {\sum_{k=n}^m a_{m-k}} \, p_{kn}\quad\mbox{for }1\leq n \leq m.
\end{align*}

Let $P$ denote the matrix of the coefficients $p_{mn}$ in \eqref{FaberP}. It takes the form
\beq\label{def:matP}
{P}:=\left[p_{mn}\right]_{m,n=0}^\infty=
\begin{bmatrix}
1 & 0 & \hskip 4mm 0 \hskip 4mm &  \hskip 4mm 0\quad \cdots\\[4mm]
-a_0 &1 & 0 &  \hskip 4mm0   \quad  \cdots\\[4mm]
a_0^2-2a_1 &-2a_0 & 1 & \hskip 4mm 0  \quad \cdots\\[4mm]
-a_0^3+3a_0a_1-3a_2&3a_0^2-3a_1 & -3a_ 0 &  \hskip 4mm1 \quad \cdots\\[4mm]
\vdots & \vdots & \vdots & \hskip 4mm \vdots\quad\ddots 
\end{bmatrix}.
\eeq

For any $M\geq 0$, the finite section $P_M:=[p_{mn}]_{m,n=0}^M$ is a $(M+1)\times(M+1)$ lower triangular matrix with ones on the diagonal, and is therefore invertible.
Moreover, the entries of $P_M^{-1}$ remain unchanged for all larger $M$. 
We thus define $P^{-1}$ as the semi-infinite matrix whose entries are the entries of $P_M^{-1}$.

\paragraph{Expansion coefficient matrix $\widetilde{D}$ for $F'_m$.} 
The derivatives of the Faber polynomial $F_m(z)$, which are polynomials of degree $m-1$, can be expanded into $F_0(z),\dots,F_{m-1}(z)$. In other words, the derivatives admit the expression:
\begin{align}
\label{Faber:deriv:coeffi}
F_m'(z) &= \sum_{k=0}^{m-1} \widetilde{d}_{m k}\, F_k(z)\quad\mbox{for }m\geq 0
\end{align}
with some coefficients $\widetilde{d}_{mk}=\widetilde{d}_{mk}(a_0,\dots,a_{m-2})$. 
For instance, $F_0'(z)=0=\widetilde{d}_{00} \, F_0(z)$ and $F_1'(z)=1=d_{10}\, F_0(z)$ where $\widetilde{d}_{00}=0$ and $\widetilde{d}_{10}=1$. For notational convenience, we set $$\widetilde{d}_{mk}=0\quad\mbox{for all }k\geq m.$$
We now introduce the matrix
\beq\label{def:mat:d}
\widetilde{ D}:=\left[\widetilde{d}_{mn}\right]_{m,n=0}^\infty=
\begin{bmatrix}
0& 0 & \hskip 1mm 0 \hskip 1mm &\hskip 1mm 0  \  \cdots\\[1mm]
1 &0 & 0 & 0   \  \cdots\\[1mm]
-2a_0 &2& 0 & 0  \ \cdots\\[1mm]
3a_1&0& 3 & 0 \ \cdots\\[1mm]
\vdots & \vdots & \vdots & \vdots\ \ddots 
\end{bmatrix}.
\eeq

For later use, we also set
$$
d_{mn}:=
 \begin{cases}
 \ds\frac{\widetilde{d}_{mn}}{m\gamma^m}\quad &\mbox{for }m\neq 0,\\
\ds 0 &\mbox{for }m=0
\end{cases}
$$
and
\beq\label{def:mat:D}
{ D}:=\left[  d_{mn}\right]_{m,n=0}^\infty=
\begin{bmatrix}
0& 0 & \hskip 1mm 0 \hskip 1mm &\hskip 1mm 0  \  \cdots\\[1mm]
d_{10} &0 & 0 & 0   \  \cdots\\[1mm]
d_{20}&d_{21}& 0 & 0  \ \cdots\\[1mm]
d_{30}&d_{31}& d_{32}& 0 \  \cdots\\[1mm]
\vdots & \vdots & \vdots & \vdots\ \ddots 
\end{bmatrix}
\eeq

\begin{lemma}\label{d_mj}
The matrices $ P$ and $\widetilde{ D}$ (see \eqref{def:matP} and \eqnref{def:mat:d}) satisfy $$
\widetilde{ D}= { PTP}^{-1}
$$
with the matrix $T$ given in \eqref{def:mat:T}.
\end{lemma}
\begin{proof}
Note that \eqref{FaberP} and \eqref{Faber:deriv:coeffi} lead to the two formulas:
\begin{align*}
F_m'(z) &= \sum_{n=0}^{m-1} \left( \sum_{k=n}^{m-1} \widetilde{d}_{mk} \, p_{kn} \right) z^n, \quad
F_m'(z) = \sum_{n=0}^{m-1} (n+1) p_{m(n+1)} z^n.
\end{align*}
By comparing the coefficients, we derive that
$$
\left[{PT}\right]_{mn}=(n+1) p_{m(n+1)} = \sum_{k=n}^{m-1} \widetilde{d}_{mk}\, p_{kn} = \sum_{k=0}^\infty \widetilde{d}_{mk}\, p_{kn} = \left[\widetilde{ D} {P}\right]_{mn}\quad \mbox{for all } m,n.
$$
This proves the lemma.
\end{proof}

\smallskip

In the following sections, we use the coordinate system via $z=\Psi(w)$, as introduced in Section \ref{subsec:pre2}. In particular, on the boundary $\p\Om$, we parametrize the point $z\in \p \Om$ as $z=\Psi(w)\in\p\Om$, where $|w|=\gamma$. 
As mentioned earlier, we assume that $\Om$ has an analytic boundary. That is, the associated exterior conformal mapping $\Psi$ admits a conformal extension to the region  $|w|>\gamma-\delta$ for some $\delta>0$.
This allows us to employ the series expansion of Faber polynomials in powers of $w^{\pm n}$ in a neighborhood of the boundary $\p\Om$ within the interior as well as in the exterior of $\Om$.

\section{Series expansions for the background loading $H$}\label{sec:deri:H}
Let us apply the decomposition \eqnref{complex:ui} to the background field $H$.
Since the Faber polynomials form a basis for complex analytic functions, the function $H$ has an expansion in an open neighborhood of $\overline{\Om}$ as the following series without the constant term:
\begin{gather}\label{H:sum}
H(z)=\sum_{m=1}^\infty H_m(z),  \
\intertext{where each term $H_m(z)$ is given by}
H_m(z)
=\kappa A_m F_m(z) - z\overline{A_m F_m'(z)} + \overline{B_m F_m(z)}
\end{gather}
with complex coefficients $A_m$ and $B_m$.
In our analysis, constant background solutions are omitted.
To simplify notation, we introduce the following semi-infinite diagonal matrices indexed from $0$ to $\infty$:
\begin{align}\label{Mat:AB}
\begin{aligned}
{A}&=\text{diag} \big(0, \, A_1,\, A_2,\, A_3,\,\cdots),\qquad
{B}=\text{diag} \big(0, \, B_1,\, B_2,\, B_3,\,\cdots).
\end{aligned}
\end{align}

We can expand $H(z)$ in powers of $w^{\pm n}$ as follows. 
\begin{theorem}\label{Thm:Mat:H}
Let $H$ be given by \eqref{H:sum} with the diagonal coefficient matrix $A$ and $B$. Let the operator $\mathcal{I}^e$ be given as in \eqref{def:I}. 
For $z=\Psi(w)\in\p\Om$, we have
\begin{align}\label{eqn:H_m:mat1}
\begin{cases}
\ds H(z)=\sum_{k=1}^\infty\, [\bm{h}^{(1)}]_{k} \, w^k+\sum_{k=0}^\infty\, [\bm{h}^{(2)}]_{k}\, w^{-k},\\
\ds \mathcal{I}^e[H](z)=\sum_{k=1}^\infty \, [\bm{h}^{(3)}]_{k}\, w^k+\sum_{k=0}^\infty\, [\bm{h}^{(4)}]_{k}\, w^{-k},
\end{cases}
\end{align}
where $\bm{h}^{(j)}$, $j=1,\dots,4$, are row vectors given by 
$$ [\bm{h}^{(j)}]_{k}=\sum_{m=1}^\infty\, [{\mathbb{H}}^{(j)}]_{mk}\quad\mbox{for each }k=0,1,2,\dots$$
and ${\mathbb{H}}^{(j)}$ are infinite matrices given by
\beq\label{eqn:H_m:mat2}
\begin{aligned}
{\mathbb{H}}^{(1)}&=\kappa A - \overline{A}\,\mathcal{N}\gamma^{\mathcal{N}}\,\overline{D} \,\big( \gamma^{2\mathcal{N}}{\Psi_{0}} +\overline{C}\,\gamma^{-2\mathcal{N}}{\Psi_{-}}\big)I_0
+ \overline{B}\,\overline{C}\gamma^{-2\mathcal{N}},\\[1mm]
{\mathbb{H}}^{(2)}&=\kappa AC - \overline{A}\,\mathcal{N}\gamma^{\mathcal{N}}\,\overline{D} \,\big( \gamma^{2\mathcal{N}}{\Psi_{-}^{\top}}+ \overline{C}\,\gamma^{-2\mathcal{N}}{\Psi_{+}} \big)
+\overline{B}\gamma^{2\mathcal{N}},\\[1mm]
{\mathbb{H}}^{(3)}&=\mu\left(
 A + \overline{A}\,\mathcal{N}\gamma^{\mathcal{N}}\,\overline{D} \,\big( \gamma^{2\mathcal{N}}{\Psi_{0}} +\overline{C}\,\gamma^{-2\mathcal{N}}{\Psi_{-}}\big)I_0
- \overline{B}\,\overline{C}\gamma^{-2\mathcal{N}}
\right),\\[1mm]
{\mathbb{H}}^{(4)}&=\mu\left(
AC + \overline{A}\,\mathcal{N}\gamma^{\mathcal{N}}\,\overline{D}\,\big( \gamma^{2\mathcal{N}}{\Psi_{-}^{\top}}+ \overline{C}\,\gamma^{-2\mathcal{N}}{\Psi_{+}} \big)I_0
- \overline{B}\gamma^{2\mathcal{N}}\right).
\end{aligned}
\eeq
\end{theorem}
\begin{proof}
Putting
\beq\label{def:I_123}
I_1(z)=A_m F_m(z),\quad I_2(z)=\overline{B_m F_m(z)}, \quad I_3(z)=z\overline{A_m\, F_m'(z)},
\eeq
it holds that, on $\p\Om$, 
\beq
\label{H_m:decomp}
\begin{cases}
H_m(z)=\kappa I_1(z) + I_2(z)-I_3(z),\\[1mm]
 \mathcal{I}^e [H_m](z)=\mu\left(I_1(z)-I_2(z)+I_3(z)\right).
\end{cases}
\eeq

Using \eqref{H_m:decomp} and the relation
\beq\label{boundary:w}
\overline{w}=\gamma^2w^{-1}\quad\mbox{for }|w|=\gamma,
\eeq
we obtain
\begin{align*}
 I_1&=  A_m \Big(w^{m}+\sum_{n=1}^{\infty}c_{mn}w^{-n}\Big)= \sum_{n=1}^\infty \left[{A}\right]_{mn}\, w^n +\sum_{n=0}^\infty \left[{A}{C}\right]_{mn}\, w^{-n} 
 \end{align*}
 and
 \begin{align*}
I_2
 &=\, \overline{B_m}\, \Big(\gamma^{2m}{w^{-m}}+\sum_{l=1}^{\infty}\overline{c_{ml}}\gamma^{-2l}w^{l}\Big)
  = \sum_{n=1}^\infty\left[\,\overline{{B}}\,\overline{{C}}\gamma^{-2\mathcal{N}} \right]_{mn}w^n+\sum_{n=0}^\infty \left[\, \overline{{B}}\,\gamma^{2\mathcal{N}} \right]_{mn}w^{-n}.
 \end{align*}
Similarly, by \eqref{Faber:deriv:coeffi}, 
 \begin{align*}
I_3&=z\overline{A_m}\sum_{j=0}^{m-1} \overline{\widetilde{d}_{m j}}\,\overline{F_j(z)}\\
&=\overline{A_m}\Psi(w)\sum_{j=0}^{m-1}\overline{\widetilde{d}_{mj}}\,  \Big(\overline{w^{j}}+\sum_{l=1}^{\infty}\overline{c_{jl}w^{-l}}\Big)\\
&=
\overline{A_m}\Big(\sum_{k=-1}^\infty a_{k}w^{-k}\Big)
 \sum_{j=0}^{m-1}\,\overline{\widetilde{d}_{mj}} \, \Big(\gamma^{2j}{w^{-j}}
+\sum_{l=1}^{\infty}\,\overline{c_{jl}}\,\gamma^{-2l}w^l\Big).
\end{align*}
By setting $a_{-k}=0$ for all $k\geq 2$, we can write
\begin{align*}
I_3
=&\sum_{n=1}^{\infty}\bigg(\overline{A_m}\,\sum_{j=0}^{\infty}\,\overline{\widetilde{d}_{mj}}\,\gamma^{2j}a_{-n-j}
+\overline{A_m}\,\sum_{j,l=0}^{\infty} \overline{ \widetilde{d}_{mj}}\,\overline{ c_{jl}}\, \gamma^{-2l}a_{l-n}
\bigg)w^{n} \\
+&\sum_{n=0}^{\infty}\bigg( \overline{A_m}\,\sum_{j=0}^{\infty}\,\overline{\widetilde{d}_{mj}}\,\gamma^{2j}a_{n-j}
+ \overline{A_m}\,\sum_{j,l=0}^{\infty}\overline{\widetilde{d}_{mj}}\, \overline{c_{jl}}\, \gamma^{-2l}a_{l+n}\bigg)w^{-n}.
\end{align*}
Using the matrices introduced in Section~\ref{subsec:matrices}, specifically \eqref{def:Psi:mat} and \eqnref{def:mat:D}, we express $I_3$ as
\begin{align*}
I_3
=\ & \sum_{n=1}^{\infty}\left[ \overline{{A}}\mathcal{N}\gamma^{\mathcal{N}}\overline{{D}}\, \Big(\gamma^{2\mathcal{N}}\Psi_0 + \overline{C}\gamma^{-2\mathcal{N}}{\Psi_-}\Big)\, \right]_{mn}w^n\\
+&\sum_{n=0}^\infty \left[\overline{{A}}\mathcal{N}\gamma^{\mathcal{N}}\overline{{D}}\, \Big( \gamma^{2\mathcal{N}}{\Psi_{-}^{\top}}+ \overline{{C}}\gamma^{-2\mathcal{N}}{\Psi_{+}}\Big)  \right]_{mn}w^{-n}.
\end{align*}
Combining this with \eqref{H_m:decomp} and \eqref{def:I_123}, the proof is complete. 
\end{proof}

\smallskip

For later use, we introduce the following notation.
\begin{notation}\label{notation:x:h}
Let ${\bm h}^{(j)}$ be the row vectors defined in Theorem~\ref{Thm:Mat:H}. 
We define the block matrix $\bm h$ of size $(1\times 8)$ as 
\begin{align}\label{def:block:H}
\bm{h} &:= \begin{bmatrix} 
\bm{h}^{(1)} &\overline{\bm{h}^{(1)} } & \bm{h}^{(2)} &\overline{\bm{h}^{(2)}} & \bm{h}^{(3)} &\overline{\bm{h}^{(3)} } & \bm{h}^{(4)} & \overline{\bm{h}^{(4)}}\,
\end{bmatrix}.
\end{align}
\end{notation}

\section{Geometric series expansion for the single-layer potential}\label{subsec:operators}

In this section, for the density functions  $\psi$ and $\phi$ in \eqref{complex:sol:u}, we derive the series expansions of the interior and exterior single-layer potentials in powers of $w^{\pm n}$. 
We can expand $\psi$ and $\phi$ into the basis $\{\varphi_{\pm n}\}_{n\in\NN\cup\{0\}}$ (see \eqref{exterior_basis}) as
\begin{align}\label{densities:general}
\begin{aligned}
\psi &= \sum_{n=1}^\infty \left( x_{n}^{{e}}\,\varphi_n  + x_{-n}^{{e}}\,\varphi_{-n}\right),\\
\phi &= \sum_{n=1}^{\infty}\left( x_{n}^{{i}} \,\varphi_n + x_{-n}^{{i}}\,\varphi_{-n}\right) + x_{0}^{{i}}\,\varphi_0
\end{aligned}
\end{align}
with complex coefficients $x_{\pm n}^e$ and $x_{\pm n}^i$ by using the fact that $\bpsi\in L^2_{\mathcal{R}}(\partial \Omega)$.

For later use, we introduce the following notation.
\begin{notation}\label{notation:x}
For the coefficients in \eqref{densities:general}, we define the row vectors
\begin{align}\label{def:vec:b}
{\bm x}^{e}_\pm&=
\begin{bmatrix}
0& x^e_{\pm 1} & x^e_{\pm 2} & x^e_{\pm 3} &\cdots
\end{bmatrix},\\
{\bm x}^{i}_+&=
\begin{bmatrix}
0 & x^i_{1} & x^i_{2} & x^i_{3}& \cdots
\end{bmatrix},\\
{\bm x}^{i}_-&=
\begin{bmatrix}
x^i_{0} & x^i_{- 1} & x^i_{- 2} & x^i_{-3}& \cdots
\end{bmatrix}.
\end{align}
We define the block row matrix $\bm x$ of size $(1\times 8)$ as 
\begin{align}\label{def:block:x}
\bm{x} := \begin{bmatrix} 
{\bm x}^e_+ &\  \overline{{\bm x}^e_+} &\  {\bm x}^e_- &\  \overline{{\bm x}^e_-} & 
\ {\bm x}^i_+ & \ \overline{{\bm x}^i_+} &\  {\bm x}^i_- & \overline{{\bm x}^i_-} \,
\end{bmatrix}.
\end{align}
\end{notation}

\subsection{Integral operators for geometric density basis}\label{subsec:integral:density}

Our analysis begins with estimating the boundary integral operators introduced in Subsection~\ref{subsec:CIE} when applied to geometric density basis $\varphi_{\pm n}$.

For $n\geq 1$, it follows from \eqref{Faber:log:expan} (see \cite{Cho:2024:ASR,Mattei:2021:EAS}) that
\begin{align}\label{Lphi:1}
\mathcal{L}[\varphi_n](z)
&=\begin{cases}
\ds -\frac{1}{n}\gamma^{-n}F_n(z)\ \quad\qquad\qquad\quad&\mbox{in }\Om,\\[3mm]
\ds -\frac{1}{n}\gamma^{-n}\left(F_n(z)-w^n\right)\ \quad\qquad\qquad\quad&\mbox{in }\CC\setminus\overline{\Om},
\end{cases}\\ \label{Lphi:2}
\mathcal{L}[\varphi_{-n}](z)
&=\begin{cases}
\ds 0\quad\qquad\qquad\qquad\qquad\qquad&\mbox{in }\Om,\\
\ds -\frac{1}{n}\gamma^n w^{-n}\quad\qquad\qquad\qquad\qquad\qquad&\mbox{in }\CC\setminus\overline{\Om}.
\end{cases}
\end{align}

Differentiating \eqref{Lphi:1} and \eqref{Lphi:2} with respect to $z$, we obtain, for $n\geq 1$,
\begin{align}\label{Cphi:1}
\mathcal{C}[\varphi_n](z)
&=\begin{cases}
\ds -\frac{1}{n}\gamma^{-n}F_n'(z)\qquad&\mbox{in }\Om,\\[3mm]
\ds -\frac{1}{n} \gamma^{-n}F_n'(z)+ \gamma^{-n}w^{n-1}\frac{1}{\Psi'(w)}\qquad&\mbox{in }\CC\setminus\overline{\Om},
\end{cases}\\ \label{Cphi:2}
\mathcal{C}[\varphi_{-n}](z)
&=\begin{cases}
\ds 0\quad\qquad\qquad\qquad\qquad&\mbox{in }\Om,\\
\ds\gamma^n w^{-n-1} \frac{1}{\Psi'(w)}\quad\qquad\qquad\qquad\qquad&\mbox{in }\CC\setminus\overline{\Om}.
\end{cases}
\end{align}
We also have
\beq\label{C:vphi0}
\mathcal{C}[\varphi_0](z)=0\quad\mbox{for }z\in\Om.
\eeq
Indeed, by \eqref{dsigma}, it holds that
\begin{align*}
\mathcal{C}[\varphi_0](z)
&=
\frac{1}{2\pi}\int_{0}^{2\pi}\frac{1}{z-\Psi(\gamma e^{\rmi\theta})}\,d\theta\quad\mbox{for }z\in\Om\\
&=\frac{1}{2\pi}\int_{0}^{2\pi}\frac{1}{z-\Psi(r e^{\rmi\theta})}\,d\theta\quad\mbox{for any }r>\gamma\\
&=\frac{1}{2\pi}\lim_{r\to\infty}\int_{0}^{2\pi}\frac{1}{z-\Psi(r e^{\rmi\theta})}\,d\theta=0
\end{align*}
Since
$$\zeta=\Psi(\gamma e^{\rmi\theta})=\sum_{k=-1}^\infty a_{k}\gamma^{-k}e^{-\rmi k \theta}\quad\mbox{for }\zeta\in\p\Om,$$
we obtain the following relation (see the proof of Lemma 3.1 in \cite{Mattei:2021:EAS} for the convergence of the right-hand side):
\beq\label{Cphi:3}
\mathcal{C}[\overline{\zeta}\varphi_{l}](z)=\sum_{k=-1}^\infty \overline{ a_{k}}\,\gamma^{-k}\,\mathcal{C}[\varphi_{k+l}](z)\quad \mbox{for all } l\in \ZZ.
\eeq

\paragraph{Decomposition of $\mathcal{C}$. } 
We can decompose $\mathcal{C}[\varphi_k]$ into two components: one that does not explicitly involve $1/\Psi'(w)$, and one that does. Specifically, for $n\in\NN$, we have
 \beq\label{C_decomp}
\mathcal{C}[\varphi_{\pm n}](z)={\mathcal{C}}_1[\varphi_{\pm n}](z)+{\mathcal{C}}_2[\varphi_{\pm n}](z),\quad z\in \CC\setminus\p\Om,
\eeq
where 
\beq\label{def:C1}
\begin{aligned}
&{\mathcal{C}}_1[\varphi_{-n}](z):=0,\\
&{\mathcal{C}}_1[\varphi_n](z):= -\frac{1}{n}\gamma^{-n}F_n'(z),\\
\end{aligned}
\eeq
and for $k=\pm n$,
\beq\label{def:C2}
\begin{aligned}
{\mathcal{C}}_2[\varphi_k](z):=
\begin{cases}
\ds 0,&z\in\Om,\\
\ds \gamma^{-k}w^{k-1}\frac{1}{\Psi'(w)},&z\in\CC\setminus\overline{\Om}.
\end{cases}
\end{aligned}
\eeq
It is important to observe that the decomposition of $\mathcal{C}$ in \eqref{C_decomp} leads to the cancellation of terms involving ${1}/{\Psi'(w)}$ in the series expansions of the single-layer potential on $\p\Om$. This cancellation plays a crucial role in deriving the matrix formulation of the elastostatic inclusion problem. Furthermore, $\mathcal{C}_1[\varphi_{\pm n}](z)$ are continuous across $\p\Om$. These observations are formalized in the following lemma:
\begin{lemma}\label{prop:C2:main}
For each $k\in\ZZ\setminus\{0\}$, the following holds. 
\begin{itemize}
\item[\rm (a)]
Continuity of $\mathcal{C}_1[\varphi_k]$:
$$\mathcal{C}_1[\varphi_k]\Big|_{\p\Om}^+=\mathcal{C}_1[\varphi_k]\Big|_{\p\Om}^-.$$
\item[\rm (b)]
Cancellation of terms involving $\mathcal{C}_2[\varphi_k]$:
$$
 -\Psi(w)\overline{\mathcal{C}_2[\varphi_k](z)} + \overline{\mathcal{C}_2[\overline{\zeta}\varphi_k](z)}\to 0\quad\mbox{as }z\to\p\Om,$$
where the limit is taken from both the exterior and interior of $\Om$.
\item[\rm(c)] Reduction of the full combination to the \(\mathcal{C}_1\) part:
$$
\left(-\Psi(w)\,\overline{\mathcal{C}[\varphi_k](z)} + \overline{\mathcal{C}[\overline{\zeta}\varphi_k](z)}\,\right)\Big|^\pm_{\p\Om}=\left(-\Psi(w)\,\overline{\mathcal{C}_1[\varphi_k](z)} + \overline{\mathcal{C}_1[\overline{\zeta}\varphi_k](z)}\,\right)\Big|_{\p\Om}.%\quad\mbox{as }z\to\p\Om,
$$
\end{itemize}
 \end{lemma}
\begin{proof}
Part (a) follows directly from the definition of  $\mathcal{C}_1[\varphi_k]$ in \eqref{def:C1}, which is continuous in $\CC$. 

To prove (b), we consider $z\in \CC\setminus\overline{\Om}$ and use the definition in \eqref{def:C2} and the relation \eqref{boundary:w}. We compute
\begin{align*}
 -\overline{\Psi(w)}\,{\mathcal{C}_2[\varphi_k](z)} + {\mathcal{C}_2[\overline{\zeta}\varphi_k](z)}
 &=-\bigg(\overline{\Psi(w)} \gamma^{-k}w^{k-1}+\sum_{j=-1}^\infty \overline{ a_{j}}\,\gamma^{-j}\,\gamma^{-j-k}w^{j+k-1}  \bigg)\frac{1}{\Psi'(w)}\\
 &=-\bigg(\overline{\Psi(w)}+\sum_{j=-1}^\infty \overline{ a_{j}}\,\gamma^{-2j}\,w^{j}\bigg)\gamma^{-k}w^{k-1}\frac{1}{\Psi'(w)}.
\end{align*}
By \eqnref{boundary:w}, we prove (b). 
Part (c) follows directly from (a) and (b). 
\end{proof}

\subsection{Exterior expansion of the single-layer potential}

For $\psi$ in \eqref{densities:general}, we represent the corresponding single-layer potential in the exterior of $\Om$ as a series of $w^{\pm n}$. 

\begin{lemma}\label{lemma:ext:L}
As $z$ tends to $\p\Om$ from the exterior of $\Om$, we have
\beq\label{ext:L:mat1}
\begin{aligned}
\ds \mathcal{L}[\psi] (z)&=-\sum_{k=0}^\infty\sum_{n=1}^\infty \Big( x^e_n \left[\mathcal{N}^{-1}_0\gamma^{-\mathcal{N}}C\right]_{nk}+x^e_{-n}\left[ \mathcal{N}^{-1}_0\gamma^{\mathcal{N}} \right]_{nk}\Big) w^{-k},\\
\ds \overline{\mathcal{L}[ \overline{\psi} ] (z)} &=-\sum_{k=1}^\infty \sum_{n=1}^\infty \Big(x^e_n \,\left[\mathcal{N}^{-1}_0\gamma^{-\mathcal{N}}\right]_{nk}
+x^e_{-n}\,[\mathcal{N}^{-1}_0\gamma^{\mathcal{N}}\overline{C}\gamma^{-2\mathcal{N}}]_{nk}\Big)w^{k}.
\end{aligned}
\eeq

\end{lemma}
\begin{proof}
By \eqref{eqn:Faberdefinition}, \eqref{Lphi:1}--\eqref{Cphi:2} and \eqref{boundary:w},  as $z$ tends to $\p\Om$ from the exterior of $\Om$, 
\begin{align*}
\mathcal{L}[\psi](z) & = 
-\sum_{n=1}^{\infty} \left( x_n^e \frac{\gamma^{-n}}{n}\sum_{k=1}^{\infty}c_{nk}w^{-k}
 + x_{-n}^e \frac{\gamma^n}{n}w^{-n}\right),
\\
\overline{\mathcal{L}[ \overline{\psi} ](z)} & = 
-\sum_{n=1}^{\infty} \left( x_{n}^e \frac{\gamma^{-n}}{n}w^{n} 
+ x_{-n}^e \frac{r^n}{n} \sum_{k=1}^{\infty}\overline{c_{nk}}\,\gamma^{-2k}w^{k}\right),
\end{align*}
which imply the matrix expression \eqref{ext:L:mat1}.
\end{proof}

Unlike $\mathcal{L}[\psi]$, the functions $\Psi(w)\overline{\mathcal{C}[\varphi](z)}$ and $\overline{\mathcal{C}[\overline{\zeta}\varphi](z)}$ involve the term $\frac{1}{\overline{\Psi'(w)}}$, whose expansion in $w^{\pm n}$ is highly nontrivial. However, this term cancels out when we substract the two functions. This allows us to derive the following lemma with a proof in Section~\ref{subsec:Proof:Lemma}. 

\begin{lemma}\label{Lem:sing:mat:ex}
As $z$ approaches $\p\Om$ from the exterior of $\Om$,  we have the relation
 \begin{align*}
 -\Psi(w)\overline{\mathcal{C}[\psi](z)} + \overline{\mathcal{C}[\overline{\zeta}\psi](z)}
 =&\sum_{k=1}^\infty\sum_{n=1}^\infty \, \Big(\overline{x^e_n}\,[\mathbb{M}^{(2,1)}]_{nk}+\overline{x^e_{-n}}\,[\mathbb{M}^{(4,1)}]_{nk}\Big) \, w^k\\
 &+\sum_{k=0}^\infty\sum_{n=1}^\infty \,  \Big(\overline{x^e_n}\,[\mathbb{M}^{(2,2)}]_{nk}+\overline{x^e_{-n}}\,[\mathbb{M}^{(4,2)}]_{nk}\Big) \, w^{-k},
 \end{align*}
with matrices given by
\begin{align*}
\mathbb{M}^{(2,1)}&=
\overline{D}\gamma^{2\mathcal{N}}{\Psi_{0}} 
+ \overline{D}\,\overline{C} \gamma^{-2\mathcal{N}}{\Psi_{-}}
- \gamma^{\mathcal{N}} {\Psi_{-}^{\top}}\,\gamma^{-\mathcal{N}} \overline{D}\,\overline{C} \gamma^{-2\mathcal{N}},
&\mathbb{M}^{(4,1)}&=-\gamma^{-\mathcal{N}}{\Psi_+}\,\gamma^{-\mathcal{N}}\overline{D}\,\overline{C}\gamma^{-2\mathcal{N}},\\
\mathbb{M}^{(2,2)}&=
\overline{D} \gamma^{2\mathcal{N}} \Psi_{-}^{\top}
+\overline{D}\, \overline{C} \gamma^{-2\mathcal{N}}{\Psi_{+}}
-\gamma^{\mathcal{N}} \Psi_-^{\top}\, \gamma^{-\mathcal{N}} \overline{D} \gamma^{2\mathcal{N}},
&\mathbb{M}^{(4,2)}&=-\gamma^{-\mathcal{N}}{\Psi_+}\,\gamma^{-\mathcal{N}}\overline{D}\gamma^{2\mathcal{N}}.
\end{align*}
The superscripts $(i,j)$ in $ \mathbb{M}^{(i,j)}$ are same with those in Theorem~\ref{Thm:Mat:Se}. 
\end{lemma}

\begin{theorem}[Exterior limit on $\p\Om$]\label{Thm:Mat:Se}
For $\psi$ given as in \eqref{densities:general}, as $z$ tends to $\p\Om$ from the exterior of $\Om$, we have
\begin{align*}
2S_{\partial \Omega}[\psi](z) =&\sum_{k=1}^\infty\sum_{n=1}^\infty \, \Big({x^e_n}\,[{\mathbb{S}}^{(1,1)}]_{nk}+\overline{x^e_n}\,[{\mathbb{S}}^{(2,1)}]_{nk}+{x^e_{-n}}\,[{\mathbb{S}}^{(3,1)}]_{nk}+\overline{x^e_{-n}}\,[{\mathbb{S}}^{(4,1)}]_{nk}\Big) \, w^k\\
 &+\sum_{k=0}^\infty\sum_{n=1}^\infty \,  \Big({x^e_n}\,[{\mathbb{S}}^{(1,2)}]_{nk}+\overline{x^e_n}\,[{\mathbb{S}}^{(2,2)}]_{nk}+{x^e_{-n}}\,[\mathbb{S}^{(3,2)}]_{nk}+\overline{x^e_{-n}}\,[\mathbb{S}^{(4,2)}]_{nk}\Big) \, w^{-k},
\end{align*}
and
\begin{align*}
\mathcal{I}^e\big[2S_{\partial \Omega}[\psi]\big](z)=&\sum_{k=1}^\infty\sum_{n=1}^\infty \, \Big({x^e_n}\,[{\mathbb{S}}^{(1,3)}]_{nk}+\overline{x^e_n}\,[{\mathbb{S}}^{(2,3)}]_{nk}+{x^e_{-n}}\,[\mathbb{S}^{(3,3)}]_{nk}+\overline{x^e_{-n}}\,[\mathbb{S}^{(4,3)}]_{nk}\Big) \, w^k\\
 &+\sum_{k=0}^\infty\sum_{n=1}^\infty \,  \Big({x^e_n}\,[{\mathbb{S}}^{(1,4)}]_{nk}+\overline{x^e_n}\,[{\mathbb{S}}^{(2,4)}]_{nk}+{x^e_{-n}}\,[\mathbb{S}^{(3,4)}]_{nk}+\overline{x^e_{-n}}\,[\mathbb{S}^{(4,4)}]_{nk}\Big) \, w^{-k}
\end{align*}
with semi-infinite matrices given by
\begin{align*}
{\mathbb{S}}^{(1,1)}&=-\alpha\mathcal{N}^{-1}_0 \gamma^{-\mathcal{N}}, &
{\mathbb{S}}^{(2,1)}&=\beta\, I_0\mathbb{M}^{(2,1)}I_0, &
{\mathbb{S}}^{(3,1)}&=-\alpha\mathcal{N}^{-1}_0\gamma^{-\mathcal{N}}\overline{C}\gamma^{-2\mathcal{N}}, &
{\mathbb{S}}^{(4,1)}&= \beta \, I_0\mathbb{M}^{(4,1)}I_0,
\\
{\mathbb{S}}^{(1,2)}&=-\alpha\mathcal{N}^{-1}_0\gamma^{-\mathcal{N}}C, &
{\mathbb{S}}^{(2,2)}&=\beta \, I_0\mathbb{M}^{(2,2)}, &
{\mathbb{S}}^{(3,2)}&=-\alpha\mathcal{N}^{-1}_0\gamma^{\mathcal{N}}, &
{\mathbb{S}}^{(4,2)}&= \beta \, I_0\mathbb{M}^{(4,2)},
\\
{\mathbb{S}}^{(1,3)}&=\mu\alpha\mathcal{N}^{-1}_0 \gamma^{-\mathcal{N}}, &
{\mathbb{S}}^{(2,3)}&=-\mu\beta \, I_0\mathbb{M}^{(2,1)}I_0, &
{\mathbb{S}}^{(3,3)}&=\mu\alpha\mathcal{N}^{-1}_0\gamma^{-\mathcal{N}}\overline{C}\gamma^{-2\mathcal{N}}, &
{\mathbb{S}}^{(4,3)}&= -\mu\beta \, I_0\mathbb{M}^{(4,1)}I_0,
\\
{\mathbb{S}}^{(1,4)}&=-\mu\beta\mathcal{N}^{-1}_0\gamma^{-\mathcal{N}}C, 
&{\mathbb{S}}^{(2,4)}&=-\mu\beta \, I_0\mathbb{M}^{(2,2)}I_0,
&{\mathbb{S}}^{(3,4)}&=-\mu\beta\mathcal{N}^{-1}_0\gamma^{\mathcal{N}},
&{\mathbb{S}}^{(4,4)}&=-\mu \beta \, I_0\mathbb{M}^{(4,2)}I_0.
\end{align*}

\end{theorem}
\begin{proof}
By Lemma \ref{lemma:Single:layer}, we have
\begin{align}\notag
\begin{cases}
\ds
2 S_{\p \Om}[\psi](z)=
\a\,\mathcal{L}[\psi](z)+\a\,\overline{\mathcal{L}[\overline{\psi}]}+\b \left(-z\,\overline{\C[\psi](z)}+\,\overline{\C[\overline{\zeta}\psi](z)}\right),\\[2mm]
\ds \mathcal{I}^e\big[2S_{\partial \Omega}[\psi]\big](z)
= \mu\left(\b\,\mathcal{L}[\psi](z)-\a\,\overline{\mathcal{L}[\overline{\psi}]}-\b \left(-z\,\overline{\C[\psi](z)}+\,\overline{\C[\overline{\zeta}\psi](z)}\right)\right)\quad\mbox{for }z\in \p\Om^+.
\end{cases}
\end{align}
By Lemmas \ref{lemma:ext:L} and \ref{Lem:sing:mat:ex}, we obtain
\begin{align*}
2 S_{\p \Om}[\psi](z)
= &-\alpha\sum_{k=0}^\infty\sum_{n=1}^\infty \Big( x^e_n \left[\mathcal{N}^{-1}_0\gamma^{-\mathcal{N}}C\right]_{nk}+x^e_{-n}\left[ \mathcal{N}^{-1}_0\gamma^{\mathcal{N}} \right]_{nk}\Big) w^{-k}\\
&-\alpha\sum_{k=1}^\infty \sum_{n=1}^\infty \Big(x^e_n \,\left[\mathcal{N}^{-1}_0\gamma^{-\mathcal{N}}\right]_{nk}
+x^e_{-n}\,[\mathcal{N}^{-1}_0\gamma^{\mathcal{N}}\overline{C}\gamma^{-2\mathcal{N}}]_{nk}\Big)w^{k}\\
&+\beta \sum_{k=1}^\infty\sum_{n=1}^\infty \, \Big(\overline{x^e_n}\,[\mathbb{M}^{(1,2)}]_{nk}+\overline{x^e_{-n}}\,[\mathbb{M}^{(1,4)}]_{nk}\Big) \, w^k\\
 &+\beta\sum_{k=0}^\infty\sum_{n=1}^\infty \,  \Big(\overline{x^e_n}\,[\mathbb{M}^{(2,2)}]_{nk}+\overline{x^e_{-n}}\,[\mathbb{M}^{(2,4)}]_{nk}\Big) \, w^{-k}
 \end{align*}
 and
 \begin{align*}
 \mathcal{I}^e\big[2S_{\partial \Omega}[\psi]\big](z)
= &-\mu\beta\sum_{k=0}^\infty\sum_{n=1}^\infty \Big( x^e_n \left[\mathcal{N}^{-1}_0\gamma^{-\mathcal{N}}C\right]_{nk}+x^e_{-n}\left[ \mathcal{N}^{-1}_0\gamma^{\mathcal{N}} \right]_{nk}\Big) w^{-k}\\
&+\mu\alpha\sum_{k=1}^\infty \sum_{n=1}^\infty \Big(x^e_n \,\left[\mathcal{N}^{-1}_0\gamma^{-\mathcal{N}}\right]_{nk}
+x^e_{-n}\left[\mathcal{N}^{-1}_0\gamma^{\mathcal{N}}\overline{C}\gamma^{-2\mathcal{N}}\right]_{nk}\Big)w^{k}\\
&-\mu\beta \sum_{k=1}^\infty\sum_{n=1}^\infty \, \Big(\overline{x^e_n}\,[\mathbb{M}^{(1,2)}]_{nk}+\overline{x^e_{-n}}\,[\mathbb{M}^{(1,4)}]_{nk}\Big) \, w^k\\
 &-\mu\beta\sum_{k=0}^\infty\sum_{n=1}^\infty \,  \Big(\overline{x^e_n}\,[\mathbb{M}^{(2,2)}]_{nk}+\overline{x^e_{-n}}\,[\mathbb{M}^{(2,4)}]_{nk}\Big) \, w^{-k}.
 \end{align*}
 This proves the theorem. 
\end{proof}

\subsection{Interior expansion of the single-layer potential}

For $\phi$ in \eqref{densities:general}, we represent the corresponding single-layer potential in the interior of $\Om$, near the boundary $\p \Om$, as a series of $w^{\pm n}$. As before, we parametrize the point $z\in \p \Om$ as $z=\Psi(w)\in\p\Om$ for $|w|=\gamma$.

\begin{lemma}\label{lem:int:L}
As $z$ tends to $\p\Om$ from the interior of $\Om$, we have
\beq\label{int:L:mat1}
\begin{aligned}
\ds \mathcal{L}[\phi](z)&=-\sum_{k=1}^\infty\sum_{n=1}^\infty x^i_n \left[\mathcal{N}^{-1}_0\gamma^{-\mathcal{N}}C\right]_{nk}\, w^{-k} -\sum_{k=1}^\infty\sum_{n=1}^\infty x^i_{n}\left[ \mathcal{N}^{-1}_0\gamma^{-\mathcal{N}} \right]_{nk}\, w^{k} + x_0\ln \gamma,\\
\ds \overline{\mathcal{L}[ \overline{\phi} ](z)} &=-\sum_{k=1}^\infty \sum_{n=1}^\infty x^i_{-n} \,\left[\mathcal{N}^{-1}_0\gamma^{\mathcal{N}}\right]_{nk} w^{-k}
-\sum_{k=1}^\infty \sum_{n=1}^\infty x^i_{-n}\,\left [\mathcal{N}^{-1}_0\gamma^{-\mathcal{N}}\overline{C}\gamma^{-2\mathcal{N}}\right]_{nk}w^{k} + x_0\ln \gamma.
\end{aligned}
\eeq
\end{lemma}
\begin{proof}
By \eqref{eqn:Faberdefinition}, \eqref{L:varphi0}, \eqref{Lphi:1}--\eqref{Cphi:2} and \eqref{boundary:w}, we obtain
\begin{align*}
\mathcal{L}[\phi](z) 
&=\sum_{n=1}^\infty x_{n}^i\Big(-\frac{\gamma^{-n}}{n}F_n(z) \Big)+ x_{0}^i\, \ln \gamma\\
&=-\sum_{n=1}^\infty x_{n}^i \frac{\gamma^{-n}}{n} w^n 
-\sum_{k=1}^\infty\sum_{n=1}^\infty x_{n}^i \frac{\gamma^{-n}}{n}c_{nk}w^{-k}+ x_{0}^i\, \ln \gamma
\end{align*}
and then
\begin{align*}
\mathcal{L}[\overline{\phi}](z) 
&=-\sum_{n=1}^\infty \overline{x_{-n}^i}\, \frac{\gamma^{-n}}{n} w^n 
-\sum_{k=1}^\infty\sum_{n=1}^\infty \overline{ x_{-n}^i}\, \frac{\gamma^{-n}}{n}c_{nk}w^{-k}+ \overline{x_{0}^i}\, \ln \gamma,\\
\overline{\mathcal{L}[ \overline{\phi} ](z)} & = 
-\sum_{n=1}^{\infty}x_{-n}^i \frac{\gamma^{n}}{n}w^{-n} 
-\sum_{k=1}^\infty\sum_{n=1}^\infty { x_{-n}^i}\, \frac{\gamma^{-n}}{n}\,\overline{c_{nk}}\,\gamma^{-2k}w^{k}+ {x_{0}^i}\, \ln \gamma,
\end{align*}
which lead the matrix expressions in \eqref{int:L:mat1}.
\end{proof}

\begin{lemma}\label{Lem:sing:mat:in}
As $z$ approaches $\p\Om$ from the interior of $\Om$,  we have
 \begin{align*}
 -\Psi(w)\overline{\mathcal{C}[\phi](z)} + \overline{\mathcal{C}[\overline{\zeta}\phi](z)}
 =&\sum_{k=1}^\infty\sum_{n=1}^\infty \, \Big(\overline{x^i_n}\,[\mathbb{M}^{(2,1)}]_{nk}+\overline{x^i_{-n}}\,[\mathbb{M}^{(4,1)}]_{nk}\Big) \, w^k\\
 &+\sum_{k=0}^\infty\sum_{n=1}^\infty \,  \Big(\overline{x^i_n}\,[\mathbb{M}^{(2,2)}]_{nk}+\overline{x^i_{-n}}\,[\mathbb{M}^{(4,2)}]_{nk}\Big) \, w^{-k}.
 \end{align*}
\end{lemma}
\begin{proof}
The result immediately follows from Lemma~\ref{prop:C2:main}\,(c) and Lemma~\ref{Lem:sing:mat:ex} by applying the same argument to the interior trace.
\end{proof}

\begin{theorem}[Interior limit on $\p\Om$]\label{Thm:Mat:Si}
For $\phi$ given as in \eqref{densities:general} as $z$ tends to $\p\Om$ from the interior of $\Om$, we have
\begin{align*}
2{\widetilde{S}}_{\partial \Omega}[\phi](z) =&\sum_{k=1}^\infty\sum_{n=1}^\infty \, \Big({x^i_n}\,[{\mathbb{\widetilde S}}^{(1,1)}]_{nk}+\overline{x^i_n}\,[{\mathbb{\widetilde S}}^{(2,1)}]_{nk}+{x^i_{-n}}\,[{\mathbb{\widetilde S}}^{(3,1)}]_{nk}+\overline{x^i_{-n}}\,[{\mathbb{\widetilde S}}^{(4,1)}]_{nk}\Big) \, w^k\\
&+\sum_{k=0}^\infty\sum_{n=1}^\infty \,  \Big({x^i_n}\,[{\mathbb{\widetilde S}}^{(1,2)}]_{nk}+\overline{x^i_n}\,[{\mathbb{\widetilde S}}^{(2,2)}]_{nk}+{x^i_{-n}}\,[\mathbb{\widetilde S}^{(3,2)}]_{nk}+\overline{x^i_{-n}}\,[\mathbb{\widetilde S}^{(4,2)}]_{nk}\Big) \, w^{-k}
\end{align*}
and
\begin{align*}
\mathcal{I}^i\big[2{\widetilde{S}}_{\partial \Omega}[\phi]\big](z)=&\sum_{k=1}^\infty\sum_{n=1}^\infty \, \Big({x^i_n}\,[{\mathbb{\widetilde S}}^{(1,3)}]_{nk}+\overline{x^i_n}\,[{\mathbb{\widetilde S}}^{(2,3)}]_{nk}+{x^i_{-n}}\,[\mathbb{\widetilde S}^{(3,3)}]_{nk}+\overline{x^i_{-n}}\,[\mathbb{\widetilde S}^{(4,3)}]_{nk}\Big) \, w^k\\
&+\sum_{k=0}^\infty\sum_{n=1}^\infty \,  \Big({x^i_n}\,[{\mathbb{\widetilde S}}^{(1,4)}]_{nk}+\overline{x^i_n}\,[{\mathbb{\widetilde S}}^{(2,4)}]_{nk}+{x^i_{-n}}\,[\mathbb{\widetilde S}^{(3,4)}]_{nk}+\overline{x^i_{-n}}\,[\mathbb{\widetilde S}^{(4,4)}]_{nk}\Big) \, w^{-k}
\end{align*}
with semi-infinite matrices given by
\begin{align*}
{\widetilde{\mathbb{S}}}^{(1,1)}&=-\ta\,\mathcal{N}^{-1}_0 \gamma^{-\mathcal{N}}, 
&{\widetilde{\mathbb{S}}}^{(2,1)}&=\tb\, I_0\mathbb{M}^{(2,1)}I_0,
&{\widetilde{\mathbb{S}}}^{(3,1)}&=-\ta \,\mathcal{N}^{-1}_0\gamma^{-\mathcal{N}}\overline{C}\gamma^{-2\mathcal{N}},
&{\widetilde{\mathbb{S}}}^{(4,1)}&= \tb \,\mathbb{M}^{(4,1)}I_0,
\\
{\widetilde{\mathbb{S}}}^{(1,2)}&=-\ta\, \mathcal{N}^{-1}_0\gamma^{-\mathcal{N}}C, 
&{\widetilde{\mathbb{S}}}^{(2,2)}&=\tb \,I_0\mathbb{M}^{(2,2)},
&{\widetilde{\mathbb{S}}}^{(3,2)}&=-\ta\,\mathcal{N}^{-1}_0\gamma^{\mathcal{N}} + c\,\mathbf{e}_0\otimes\mathbf{e}_0, 
&{\widetilde{\mathbb{S}}}^{(4,2)}&= \tb \,\mathbb{M}^{(4,2)},
\\[2mm]
{\widetilde{\mathbb{S}}}^{(1,3)}&=-\tmu\tb\,\mathcal{N}^{-1}_0 \gamma^{-\mathcal{N}}, 
&{\widetilde{\mathbb{S}}}^{(2,3)}&=-\tmu\tb \,I_0\mathbb{M}^{(2,1)}I_0, 
&{\widetilde{\mathbb{S}}}^{(3,3)}&=\tmu\ta\,\mathcal{N}^{-1}_0\gamma^{-\mathcal{N}}\overline{C}\gamma^{-2\mathcal{N}}, 
&{\widetilde{\mathbb{S}}}^{(4,3)}&=- \tmu\tb \,\mathbb{M}^{(4,1)}I_0,
\\
{\widetilde{\mathbb{S}}}^{(1,4)}&=-\tmu\tb\mathcal{N}^{-1}_0\gamma^{-\mathcal{N}}C, 
&{\widetilde{\mathbb{S}}}^{(2,4)}&=-\tmu\tb \,I_0\mathbb{M}^{(2,2)}I_0, 
&{\widetilde{\mathbb{S}}}^{(3,4)}&=\tmu\ta\,\mathcal{N}^{-1}_0\gamma^{\mathcal{N}}, 
&{\widetilde{\mathbb{S}}}^{(4,4)}&=-\tmu \tb \,\mathbb{M}^{(4,2)}I_0
\end{align*}
with $c=2\tilde{\alpha}\ln\gamma-\tilde{\beta}$ and $\mathbf{e}_0=[1,0,0,\dots]$. 
\end{theorem}

\begin{proof}
By Lemma \ref{lemma:Single:layer}, we have
\begin{align}\notag
\begin{cases}
\ds
2 \widetilde{S}_{\p \Om}[\phi](z)=
\ta\,\mathcal{L}[\phi](z)+\ta\,\overline{\mathcal{L}[\overline{\phi}]}+\tb \left(-z\,\overline{\C[\phi](z)}+\,\overline{\C[\overline{\zeta}\phi](z)}\right)-c_\phi,\\[2mm]
\ds \mathcal{I}^i\big[2\widetilde{S}_{\partial \Omega}[\phi]\big](z)
= \tmu\left(\tb\,\mathcal{L}[\phi](z)-\ta\,\overline{\mathcal{L}[\overline{\phi}]}-\tb \left(-z\,\overline{\C[\phi](z)}+\,\overline{\C[\overline{\zeta}\phi](z)}\right)\right)\qquad\mbox{for }z\in \p\Om^-.
\end{cases}
\end{align}
By Lemma \ref{prop:C2:main}, $2 \widetilde{S}_{\p \Om}[\phi](z)$ admits the same expansion as $2 S_{\p \Om}[\psi](z)$ except the constant term $c_\phi $, where $\alpha$ and $\beta$ are replaced by $\ta$ and $\tb$. 

By Lemmas \ref{lem:int:L} and \ref{Lem:sing:mat:in}, we have
 \begin{align*}
 \mathcal{I}^i\big[2\widetilde{S}_{\partial \Omega}[\phi]\big](z)
= &-\tmu\sum_{k=0}^\infty\sum_{n=1}^\infty \Big( x^i_n\, \tb\left[\mathcal{N}^{-1}_0\gamma^{-\mathcal{N}}C\right]_{nk}-x^i_{-n}\, \ta\left[ \mathcal{N}^{-1}_0\gamma^{\mathcal{N}} \right]_{nk}\Big) w^{-k}\\
&-\tmu\sum_{k=1}^\infty \sum_{n=1}^\infty \Big(x^i_n \,\tb\left[\mathcal{N}^{-1}_0\gamma^{-\mathcal{N}}\right]_{nk}
-x^i_{-n}\,\ta \left[\mathcal{N}^{-1}_0\gamma^{\mathcal{N}}\overline{C}\gamma^{-2\mathcal{N}}\right]_{nk}\Big)w^{k}\\
&-\tmu\tb \sum_{k=1}^\infty\sum_{n=1}^\infty \, \Big(\overline{x^i_n}\,[\mathbb{M}^{(1,2)}]_{nk}+\overline{x^i_{-n}}\,[\mathbb{M}^{(1,4)}]_{nk}\Big) \, w^k\\
 &-\tmu\tb\sum_{k=0}^\infty\sum_{n=1}^\infty \,  \Big(\overline{x^i_n}\,[\mathbb{M}^{(2,2)}]_{nk}+\overline{x^i_{-n}}\,[\mathbb{M}^{(2,4)}]_{nk}\Big) \, w^{-k}
 \end{align*}
and this completes the proof.

\end{proof}

\subsection{Proof of Lemma \ref{Lem:sing:mat:ex}} \label{subsec:Proof:Lemma}

We set $\psi$ as in \eqref{densities:general}.
By \eqref{C_decomp} and Proposition \ref{prop:C2:main}, we get
\begin{align*}
-\Psi(w)\overline{\mathcal{C}[\psi]} + \overline{\mathcal{C}[\overline{\zeta}\psi](z)} 
&=-\Psi(w)\overline{\mathcal{C}_1[\psi]} + \overline{\mathcal{C}_1[\overline{\zeta}\psi](z)} \\
&=\Psi(w)\sum_{n=1}^\infty \overline{ x^e_{n}}\frac{1}{n \gamma^n}\overline{F_n'(z)}
-\sum_{n=1}^{\infty}\overline{x_{n}^e}\sum_{k=-1}^{\infty} \dfrac{a_k}{(n+k)\gamma^{n+2k}}\overline{F_{n+k}'(z)}\\
&\quad -\sum_{n=1}^{\infty}\overline{x_{-n}^e} \sum_{k=n+1}^{\infty} \dfrac{a_k}{(k-n)\gamma^{2k-n}}\overline{F_{k-n}'(z)}\\
&=:{I}+II+ III.
\end{align*}
Let us expand $I,II$, and $III$ separately. 

Using \eqref{eqn:Faberdefinition}, \eqref{Faber:deriv:coeffi}, \eqref{def:mat:D}, and \eqref{boundary:w}, we obtain
\begin{align*}
I%=&\Psi(w)\sum_{n=1}^\infty \overline{ b^e_{m,-n}}\frac{1}{n \gamma^n}\overline{F_n'(z)}\\
=&\bigg(\sum_{k=-1}^\infty a_k w^{-k}\bigg)\sum_{n=1}^\infty \overline{ x^e_{n}}\,\frac{1}{n \gamma^n}\bigg(\sum_{j=0}^{n-1} \overline{\widetilde{d}_{nj}} \,\overline{F_j(z)}\bigg)\\
=\,&\sum_{k=-1}^{\infty}\sum_{n=1}^{\infty}\sum_{j=0}^{n-1}a_k \, \overline{ x^e_{n}} \,\frac{\overline{\widetilde{d}_{nj}}}{n \gamma^n}  \,w^{-k}
\bigg( \gamma^{2j}w^{-j} + \sum_{l=1}^{\infty} \overline{c_{jl}}\gamma^{-2l}w^{l}\bigg)
\\
=\, & \sum_{k=1}^{\infty}\bigg(\sum_{j=0}^{\infty}\sum_{n=1}^{\infty} \overline{x_{n}^{e}}\, \overline{d_{nj}}\, \gamma^{2j}a_{-k-j}\bigg)w^{k}
+\sum_{k=0}^{\infty}\bigg(\sum_{j=0}^{\infty}\sum_{n=1}^{\infty} \overline{x_{n}^{e}}\, \overline{d_{nj}}\gamma^{2j}a_{k-j}\bigg)w^{-k}
\\
&+\sum_{k=1}^{\infty}\bigg(\sum_{l=0}^{\infty}\sum_{j=0}^{\infty}\sum_{n=1}^{\infty}\, \overline{x_{n}^{e}}\, \overline{d_{nj}}\,\overline{c_{jl}}\gamma^{-2l}a_{-k+l}\bigg)w^{k}
+\sum_{k=0}^{\infty}\bigg(\sum_{l=0}^{\infty}\sum_{j=0}^{\infty}\sum_{n=1}^{\infty}\,\overline{x_{n}^{e}}\, \overline{d_{nj}}\, \,\overline{c_{jl}}\gamma^{-2l}a_{k+l}\bigg)w^{-k}
\end{align*}
and it follows that
\beq\label{prop:eqn:1}
\begin{aligned}
 I=\, &\sum_{k=0}^\infty\sum_{n=1}^\infty \,\overline{x^e_{n}}\, \Big[\, \overline{D}\,\gamma^{2\mathcal{N}}\Psi_0
+ \overline{D}\,\overline{C}\,\gamma^{-2\mathcal{N}}{\Psi_-}\Big]_{nk} w^k\\
 &+\sum_{k=1}^\infty\sum_{n=1}^\infty \,  \overline{x^e_{n}} \,\Big[\,\overline{D}\,\gamma^{2\mathcal{N}}\Psi_-^{\top}
+ \overline{D}\,\overline{C}\,\gamma^{-2\mathcal{N}}{\Psi_+}\Big]_{nk} w^{-k}.
 \end{aligned}
\eeq
Similarly, we derive
\begin{align*}
II
&= 
-\sum_{n=1}^{\infty}\overline{x_{n}^e}\sum_{k=-1}^{\infty} \dfrac{a_k}{(n+k)\gamma^{n+2k}}\,
\bigg(\sum_{j=0}^{n+k-1} \overline{\widetilde{d}_{(n+k)j}} \,\overline{F_j(z)}\bigg)\\
&=-\sum_{n=1}^{\infty}\sum_{k=-1}^{\infty}\sum_{j=0}^{n+k-1}\,\overline{x_{n}^e}\,\frac{a_k}{\gamma^k}
\,\overline{d_{(n+k)j}}\bigg(\gamma^{2j}w^{-j}+\sum_{l=1}^{\infty}\overline{c_{jl}}\,\gamma^{-2l}w^{l}\bigg)
\\
&=-\sum_{k=1}^\infty \bigg(\sum_{l=0}^\infty\sum_{j=0}^\infty\sum_{n=1}^\infty \overline{x_{n}^{e}}\,\dfrac{a_{j-n}}{\gamma^{j-n}}\, {\overline{d_{js}}}\,\overline{c_{sk}}\gamma^{-2k}\bigg)w^{k}
-\sum_{k=0}^\infty\bigg(\sum_{j=0}^\infty \sum_{n=1}^\infty \overline{x_{n}^{e}}\,\dfrac{a_{j-n}}{\gamma^{j-n}}\,{\overline{d_{jk}}}\,\gamma^{2k} \bigg)w^{-k}
\end{align*}
and it follows that
\begin{align}\label{prop:eqn:2}
 II=-&\sum_{k=0}^\infty\sum_{n=1}^\infty \,\overline{x^e_{n}}\, \Big[ \gamma^{\mathcal{N}}\,\Psi_-^{\top}\,\gamma^{-\mathcal{N}}\overline{D}\,\overline{C}\gamma^{-2\mathcal{N}}  \Big]_{nk} w^k
 -\sum_{k=1}^\infty\sum_{n=1}^\infty \,  \overline{x^e_{n}} \,\Big[\gamma^{\mathcal{N}}\,\Psi_-^{\top}\,\gamma^{-\mathcal{N}}\,\overline{D}\,\gamma^{2\mathcal{N}} \Big]_{nk} w^{-k}.
\end{align}
Finally, the last component has the form
\begin{align*}
III
&=-\sum_{n=1}^{\infty}\overline{x_{-n}^e} \sum_{k=n+1}^{\infty} \dfrac{a_k}{(k-n)\gamma^{2k-n}}\,\bigg(\sum_{j=0}^{k-n-1}\,\overline{{\widetilde{d}}_{(k-n)j}}\,\overline{F_j(z)}\bigg)\\
&=-\sum_{n=1}^{\infty} \sum_{k=n+1}^{\infty}\sum_{j=0}^{k-n-1}\overline{x_{-n}^e}\, \dfrac{a_k}{\gamma^k}\,\overline{d_{(k-n)j}}\,\bigg(\gamma^{2j}w^{-j}+\sum_{l=1}^{\infty}\overline{c_{jl}}\,\gamma^{-2l}w^{l}\bigg)\\
&=-\sum_{k=1}^\infty \bigg( \sum_{l=0}^\infty\sum_{j=0}^\infty\sum_{n=1}^\infty \overline{x_{-n}^{e}} \,\frac{a_{j+n}}{\gamma^{j+n}}\,\overline{d_{jk}}\,\overline{c_{sk}}\,\gamma^{-2k}  \bigg)w^{k}
-\sum_{k=0}^\infty\bigg(\sum_{j=0}^\infty \sum_{n=1}^\infty \overline{x_{-n}^e}\,\frac{a_{j+n}}{\gamma^{j+n}}\,\overline{d_{jk}}\,\gamma^{2k}\bigg)w^{-k},
\end{align*}
which leads
\begin{align}\label{prop:eqn:3}
 III=-&\sum_{k=0}^\infty\sum_{n=1}^\infty \,\overline{x^e_{-n}}\, \Big[  \gamma^{-\mathcal{N}}\,\Psi_+\,\gamma^{-\mathcal{N}}\overline{D}\,\overline{C}\gamma^{-2\mathcal{N}}  \Big]_{nk} w^k
 -\sum_{k=1}^\infty\sum_{n=1}^\infty \,  \overline{x^e_{-n}} \,\Big[\gamma^{-\mathcal{N}}\,\Psi_+\,\gamma^{-\mathcal{N}}\,\overline{D}\,\gamma^{2\mathcal{N}}\Big]_{nk} w^{-k}.
\end{align}
Combining \eqref{prop:eqn:1}--\eqref{prop:eqn:3}, we complete the proof. 

\qed

\section{Matrix formulation for the elastostatic inclusion problem}
In this section, we derive a matrix formulation for the elastostatic inclusion problem \eqref{eqn:main:trans} using the exterior conformal mapping. The formulation consists of the parameter $\gamma$ and the conformal mapping coefficients $a_n$.

\subsection{Main results}\label{subsec:main_results}

By \eqref{complex:sol:u}, we have for $z\in\p\Om$, 
\begin{align*}
\begin{cases}
\ds 2 S_{\p \Om}[\psi](z)-2{\widetilde{S}}_{\partial \Omega}[\phi](z)=-2H(z),\\[2mm]
\ds \mathcal{I}^e\big[2S_{\partial \Omega}[\psi]\big](z)-\mathcal{I}^i\big[2{\widetilde{S}}_{\partial \Omega}[\phi]\big](z)=-2 \, \mathcal{I}^e[H](z)+\mbox{constant}.
\end{cases}
\end{align*}
Then, it follows from Theorems \ref{Thm:Mat:H}, \ref{Thm:Mat:Se} and \ref{Thm:Mat:Si} that
\begin{align*}
-2[\bm{h}^{(j)}]_{k}=\ &\sum_{n=1}^\infty \, \Big({x^e_n}\,[{\mathbb{S}}^{(j,1)}]_{nk}+\overline{x^e_n}\,[{\mathbb{S}}^{(j,2)}]_{nk}+{x^e_{-n}}\,[{\mathbb{S}}^{(j,3)}]_{nk}+\overline{x^e_{-n}}\,[{\mathbb{S}}^{(j,4)}]_{nk}\Big)\\
&-\sum_{n=1}^\infty \, \Big({x^i_n}\,[{\mathbb{\widetilde S}}^{(j,1)}]_{nk}+\overline{x^i_n}\,[{\mathbb{\widetilde S}}^{(j,2)}]_{nk}+{x^i_{-n}}\,[{\mathbb{\widetilde S}}^{(j,3)}]_{nk}+\overline{x^i_{-n}}\,[{\mathbb{\widetilde S}}^{(j,4)}]_{nk}\Big),\quad j=1,2,3,4, \ k\in \NN.
\end{align*}
For $j=2$ and $k=0$, we have
\begin{align*}
-2[\bm{h}^{(2)}]_0&=
\sum_{n=1}^\infty \, \Big({x^e_n}\,[{\mathbb{S}}^{(1,2)}]_{n0}+\overline{x^e_n}\,[{\mathbb{S}}^{(2,2)}]_{n0}+{x^e_{-n}}\,[{\mathbb{S}}^{(3,2)}]_{n0}+\overline{x^e_{-n}}\,[{\mathbb{S}}^{(4,2)}]_{n0}\Big)\\
&\quad
-\sum_{n=1}^\infty \, \Big({x^i_n}\,[{\mathbb{\widetilde{S}}}^{(1,2)}]_{n0}+\overline{x^i_n}\,[{\mathbb{\widetilde{S}}}^{(2,2)}]_{n0}+{x^i_{-n}}\,[{\mathbb{\widetilde{S}}}^{(3,2)}]_{n0}+\overline{x^i_{-n}}\,[{\mathbb{\widetilde{S}}}^{(4,2)}]_{n0}\Big) - 2\ta x^i_0\ln\gamma + \tb x^i_0,
\end{align*}
From the row vectors ${\bm x}^{e}_\pm$, ${\bm x}^{i}_\pm$ and ${\bm h}^{(j)}$ defined in \eqref{def:block:x} and \eqref{def:block:H}, we have
 \begin{align*}
-2{\bm h}^{(j)} = & {\bm x}^e_+\, {\mathbb{S}}^{(1,j)}+\overline{ {\bm x}^e_+ }\, {\mathbb{S}}^{(2,j)}
+ {\bm x}^e_- \, {\mathbb{S}}^{(3,j)}+ \overline{ {\bm x}^e_- }\,{\mathbb{S}}^{(4,j)}\\
&-{\bm x}^i_+\, {\mathbb{\widetilde S}}^{(1,j)}-\overline{ {\bm x}^i_+ }\,{\mathbb{\widetilde S}}^{(2,j)}
- {\bm x}^i_- \, {\mathbb{\widetilde S}}^{(3,j)}- \overline{ {\bm x}^i_- }\,{\mathbb{\widetilde S}}^{(4,j)},\quad j=1,2,3,4,
\end{align*}
and, by taking the complex conjugate,
\begin{align*}
-2\, \overline{{\bm h}^{(j)}}
=&{\bm x}^e_+\,\overline{ {\mathbb{S}}^{(2,j)}}+\overline{ {\bm x}^e_+ }\, \overline{{\mathbb{S}}^{(1,j)}}
+ {\bm x}^e_- \, \overline{{\mathbb{S}}^{(4,j)}}+ \overline{ {\bm x}^e_- }\, \overline{{\mathbb{S}}^{(3,j)}}\\[1mm]
&-{\bm x}^i_+\,\overline{{\mathbb{\widetilde S}}^{(2,j)}}-\overline{ {\bm x}^i_+ }\, \overline{\widetilde{\mathbb{S}}^{(1,j)}}
- {\bm x}^i_- \, \overline{\widetilde{\mathbb{S}}^{(4,j)}}- \overline{ {\bm x}^i_- }\, \overline{\widetilde{\mathbb{S}}^{(3,j)}},\quad j=1,2,3,4.
\end{align*}
We can rewrite these relations in a block-matrix form as follows. 
\begin{theorem}\label{thm:main:modi}
Let $H$ be a given background loading, and let the corresponding density functions \(\psi\) and \(\phi\) in \eqref{complex:sol:u} have the expansion \eqref{densities:general}. Following Notation~\ref{notation:x:h} and Notation~\ref{notation:x}, let \(\bm{x}\) denote the $(1 \times 8)$ block row vector consisting of the coefficients of \(\psi\) and \(\phi\) given by \eqref{def:block:x}, and let \(\bm{h}\) denote the $(1 \times 4)$ block row vector associated with \(H\) given by \eqref{def:block:H}.
Then $\bm x$ satisfies the infinite-dimensional linear system
\beq\label{eqn:main:matrix}
\bm{x}\mathbb{E} = -2\bm{h},
\eeq
equivalently,
\beq
{\mathbb{E}}^{T}\,{\bm x}^T=-2{\bm h}^T,
\eeq
where $\mathbb{E}$ is the $(8\times 8)$ block matrix defined by
\begin{align*}
\mathbb{E}&=
\left(
\begin{array}{cccc:cccc}
{\mathbb{S}}^{(1,1)} & \overline{ {\mathbb{S}}^{(2,1)}} & {\mathbb{S}}^{(1,2)} & \overline{ {\mathbb{S}}^{(2,2)}} & {\mathbb{S}}^{(1,3)} & \overline{ {\mathbb{S}}^{(2,3)}} & {\mathbb{S}}^{(1,4)} & \overline{ {\mathbb{S}}^{(2,4)}} \\[2mm]
{\mathbb{S}}^{(2,1)} & \overline{ {\mathbb{S}}^{(1,1)}} & {\mathbb{S}}^{(2,2)} & \overline{ {\mathbb{S}}^{(1,2)}} & {\mathbb{S}}^{(2,3)} & \overline{ {\mathbb{S}}^{(1,3)}} & {\mathbb{S}}^{(2,4)} & \overline{ {\mathbb{S}}^{(1,4)}} \\[2mm]
{\mathbb{S}}^{(3,1)} & \overline{ {\mathbb{S}}^{(4,1)}} & {\mathbb{S}}^{(3,2)} & \overline{ {\mathbb{S}}^{(4,2)}} & {\mathbb{S}}^{(3,3)} & \overline{ {\mathbb{S}}^{(4,3)}} & {\mathbb{S}}^{(3,4)} & \overline{ {\mathbb{S}}^{(4,4)}} \\[2mm]
{\mathbb{S}}^{(4,1)} & \overline{ {\mathbb{S}}^{(3,1)}} & {\mathbb{S}}^{(4,2)} & \overline{ {\mathbb{S}}^{(3,2)}} & {\mathbb{S}}^{(4,3)} & \overline{ {\mathbb{S}}^{(3,3)}} & {\mathbb{S}}^{(4,4)} & \overline{ {\mathbb{S}}^{(3,4)}} \\[3mm]
\hdashline
\noalign{\vskip 9pt} 
-\,{\mathbb{\widetilde S}}^{(1,1)} & -\, \overline{{\mathbb{\widetilde S}}^{(2,1)}} & -\,{\mathbb{\widetilde S}}^{(1,2)} & -\, \overline{{\mathbb{\widetilde S}}^{(2,2)}} & -\,{\mathbb{\widetilde S}}^{(1,3)} & -\, \overline{{\mathbb{\widetilde S}}^{(2,3)}} & -\,{\mathbb{\widetilde S}}^{(1,4)} & -\, \overline{{\mathbb{\widetilde S}}^{(2,4)}} \\[2mm]
-\,{\mathbb{\widetilde S}}^{(2,1)} & -\, \overline{{\mathbb{\widetilde S}}^{(1,1)}} & -\,{\mathbb{\widetilde S}}^{(2,2)} & -\, \overline{{\mathbb{\widetilde S}}^{(1,2)}} & -\,{\mathbb{\widetilde S}}^{(2,3)} & -\, \overline{{\mathbb{\widetilde S}}^{(1,3)}} & -\,{\mathbb{\widetilde S}}^{(2,4)} & -\, \overline{{\mathbb{\widetilde S}}^{(1,4)}} \\[2mm]
-\,{\mathbb{\widetilde S}}^{(3,1)} & -\, \overline{{\mathbb{\widetilde S}}^{(4,1)}} & -\,{\mathbb{\widetilde S}}^{(3,2)} & -\, \overline{{\mathbb{\widetilde S}}^{(4,2)}} & -\,{\mathbb{\widetilde S}}^{(3,3)} & -\, \overline{{\mathbb{\widetilde S}}^{(4,3)}} & -\,{\mathbb{\widetilde S}}^{(3,4)} & -\, \overline{{\mathbb{\widetilde S}}^{(4,4)}} \\[2mm]
-\,{\mathbb{\widetilde S}}^{(4,1)} & -\, \overline{{\mathbb{\widetilde S}}^{(3,1)}} & -\,{\mathbb{\widetilde S}}^{(4,2)} & -\, \overline{{\mathbb{\widetilde S}}^{(3,2)}} & -\,{\mathbb{\widetilde S}}^{(4,3)} & -\, \overline{{\mathbb{\widetilde S}}^{(3,3)}} & -\,{\mathbb{\widetilde S}}^{(4,4)} & -\, \overline{{\mathbb{\widetilde S}}^{(3,4)}} 
\end{array}
\right).
\end{align*}
Here, the matrices ${\mathbb{S}}^{(j,l)}$ and $\widetilde{{\mathbb{S}}}^{(j,l)}$ are given in Theorem \ref{Thm:Mat:Se} and Theorem \ref{Thm:Mat:Si}. 

In the block matrix $\mathbb{E}$, the upper four rows correspond to the transmission condition for $\mathbf{u}$, and the lower four rows correspond to that for the conormal derivatives. The first four columns represent the exterior single-layer potential, and the remaining four columns represent the exterior single-layer potential. Each $(2k)$-th row is the complex conjugate of the preceding row. 
\end{theorem}

\paragraph{Geometric series solution to the transmission problem \eqref{eqn:main:trans}.} 
Given a solution $\bm x$ to \eqref{eqn:main:matrix}, we obtain the coefficients of the density functions $\psi$ and $\phi$ appearing in the expansion \ref{densities:general}.
Then, by applying Lemma~\ref{lemma:S:complex:all} along with the integral computations in Section~\ref{subsec:integral:density}, one can construct explicit series solutions to \eqref{eqn:main:trans} in terms of the coordinates defined by the exterior conformal mapping; see theorem 3.2 in \cite{Mattei:2021:EAS} for the case of a rigid inclusion.

 The rigorous justification of the convergence of the series solution and the interchangeability of summation orders in the derivations in Sections~\ref{sec:deri:H} and \ref{subsec:operators}, which depends on the decay rate of the conformal mapping coefficients $a_k$ as $k$ increases, remains for further analysis.

\subsection{Limiting cases of holes}
We now consider the limiting case of $\tlambda=\tmu=0$ (in other words, $\Om$ is a hole), then the corresponding problem is \begin{align}\label{eqn:main:trans:cavity}
\begin{cases}
\mathcal{L}_{\lambda,\mu} \mathbf{u} = 0&\text{in }\RR^2\setminus\overline{\Om},\\
\dfrac{\partial \mathbf{u}}{\partial \nu}\Big|^+ = 0 &\text{on } \partial \Om,\\
(\mathbf{u-H})(x) = O(|x|^{-1}) &\text{as }x\to +\infty.
\end{cases}
\end{align}
The solution has the representation \cite[Chapter 2]{Ammari:2015:MME}
$$\mathbf{u}(x) =\mathbf{H}(x) + \mathcal{S}_{\partial \Omega}[\bpsi_0](x), \quad 
x\in\mathbb{R}^2\setminus\overline{\Omega},$$where $\bpsi_0\in L^2_{\mathcal{R}}(\partial \Omega)$ satisfies the boundary condition on $\p\Om$ of \eqref{eqn:main:trans:cavity}. 

We can simplify the matrix formulation in Theorem \ref{thm:main:modi}, by using only the conormal derivative condition on $\p\Om$. 
The coefficients of the density function $\bpsi_0$ satisfy the matrix relation
\beq\notag
{\mathbb{E}}_0^T\,{\bm x}_0^T=-2{\bm h}_0^T,
\eeq
where ${\bm x}_0$ and ${\bm h}_0$ are $(1\times 4)$ block row vectors defined as
\begin{align*}
{\bm{x}}_0 &:= \begin{bmatrix} 
{\bm x}^e_+ &\  \overline{{\bm x}^e_+} &\  {\bm x}^e_- &\  \overline{{\bm x}^e_-}  \,
\end{bmatrix},\\
{\bm{h}}_0 &:= \begin{bmatrix} 
\bm{h}^{(3)} &\overline{\bm{h}^{(3)} } & \bm{h}^{(4)} & \overline{\bm{h}^{(4)}}\,
\end{bmatrix}.
\end{align*}
Here, ${\mathbb{E}}_0$ denotes the part of the matrix ${\mathbb{E}}$ corresponding to the conformal derivative of the exterior single-layer potential, and is given by
\begin{align*}
{\mathbb{E}}_0^{T} &=
\left(\
\begin{array}{cccc}
\left({\mathbb{S}}^{(1,3)}\right)^T & \left({\mathbb{S}}^{(2,3)}\right)^T & \left({\mathbb{S}}^{(3,3)}\right)^T & \left({\mathbb{S}}^{(4,3)}\right)^T \\[2mm]
\left(\overline{{\mathbb{S}}^{(2,3)}}\right)^T & \left(\overline{{\mathbb{S}}^{(1,3)}}\right)^T & \left(\overline{{\mathbb{S}}^{(4,3)}}\right)^T & \left(\overline{{\mathbb{S}}^{(3,3)}}\right)^T \\[2mm]
\left({\mathbb{S}}^{(1,4)}\right)^T & \left({\mathbb{S}}^{(2,4)}\right)^T & \left({\mathbb{S}}^{(3,4)}\right)^T & \left({\mathbb{S}}^{(4,4)}\right)^T \\[2mm]
\left(\overline{{\mathbb{S}}^{(2,4)}}\right)^T & \left(\overline{{\mathbb{S}}^{(1,4)}}\right)^T & \left(\overline{{\mathbb{S}}^{(4,4)}}\right)^T & \left(\overline{{\mathbb{S}}^{(3,4)}}\right)^T
\end{array}\
\right).
\end{align*}

In the following examples, we interpret the matrix representation corresponding to the background solution given by $H(z) = \overline{B_m F_m(z)}$.
\begin{example}[Disk]
We consider a disk cavity with
\begin{align*}
\Psi(w) = w + 0.5.
\end{align*}
Then, all terms vanish except for $x^{-}_{m}$ and
\begin{align*}
\bmat{
-\dfrac{\alpha}{m\gamma^m} & 0 & 0 & 0\\
0 & -\dfrac{\alpha}{m\gamma^m} &  0 & 0\\
0 & 0 & \dfrac{\beta\gamma^m}{m}&0\\
0 & 0 & 0&\dfrac{\beta\gamma^m}{m}
}
\bmat{
x_{m}^e\\[2mm] 
\overline{x_{m}^e}\\[2mm]
x_{m}^i\\[2mm] 
\overline{x_{m}^i}
}
=
-2\bmat{
0\\[2mm]
0\\[2mm]
\overline{B_m}\gamma^{2m}\\[2mm]
{B_m}\gamma^{2m}
}.
\end{align*}
Hence,
\begin{align*}
x^{e}_m = 0,\qquad x^{i}_{m} = -\frac{2\overline{B_m} m \gamma^m}{\beta}\qquad\mbox{for }m\in\mathbb{N}.
\end{align*}
These results are same with those in \cite{Cho:2024:ASR} as $\tilde{\mu}\to 0 $.
\end{example}

\begin{example}[Ellipse]
We consider an elliptical cavity with
\begin{align*}
\Psi(w) = w + 0.5 + \frac{0.3}{w}.
\end{align*}
In this case, all $x_n^e$ and $x_n^i$ vanish for $n \in \mathbb{N}$ except when $n = m$ and
\begin{align*}
&\bmat{
-\dfrac{\alpha}{m\gamma^m} & \beta\overline{a_1^{m-1}}\gamma^{-3m-2}(\gamma^4 - |a_1|^2) & -\alpha\dfrac{\overline{a_1^m}}{m\gamma^{3m}} & 0 \\
\beta{a_1^{m-1}}\gamma^{-3m-2}(\gamma^4 - |a_1|^2) & -\dfrac{\alpha}{m\gamma^m} & 0 & -\alpha\dfrac{{a_1^m}}{m\gamma^{3m}} \\
\beta\dfrac{a_1^m}{m\gamma^m} & 0 & \beta\dfrac{\gamma^m}{m} & 0 \\
0 & \beta\dfrac{\overline{a_1^m}}{m\gamma^m} & 0 & \beta\dfrac{\gamma^m}{m}
}
\bmat{
x_{m}^e \\[2mm]
\overline{x_{m}^e} \\[2mm]
x_{m}^i \\[2mm]
\overline{x_{m}^i}
}
\\&\qquad
=
-2\bmat{
\overline{B_m a_1^m}\gamma^{-2m}\\[2mm]
{B_m a_1^m}\gamma^{-2m}\\[2mm]
\overline{B_m}\gamma^{2m}\\[2mm]
{B_m}\gamma^{2m}
}.
\end{align*}
In particular, when $m=1$, we have
\begin{align*}
x_{1}^e
&= \dfrac{2\gamma^3(\lambda+2\mu)\left[(\lambda+\mu)(B_1^2a_1^2 + |B_1a_1|^2) + 2\mu|B_1a_1|^2\right]}{B_1a_1(\lambda+\mu)(\gamma^4-|a_1|^2)},\\[2mm]
x_{1}^i
&= -\dfrac{2\gamma(\lambda+2\mu)\left[(\lambda+\mu)(B_1^2a_1^2 + |B_1a_1|^2) + 2\mu|B_1|^2\gamma^4\right]}{B_1(\lambda+\mu)(\gamma^4-|a_1|^2)}.
\end{align*}
\end{example}

\section{Conclusion}

The plane elastostatic inclusion problem is a classical topic in applied mechanics and has been extensively studied because of its various applications. The layer potential technique introduces a convenient solution ansatz by reformulating the original system of partial differential equations into an equivalent boundary integral equation on the inclusion boundary. In two dimensions, the complex formulation offers an alternative framework in terms of complex holomorphic functions. 
Additionally,  the geometric function theory provides powerful tools to handle the transmission condition on the inclusion boundary, using a coordinate system introduced by the exterior conformal mapping, assuming the inclusion is a simply connected and bounded.

In this paper, we combine the layer potential technique, complex formulation, and geometric function theory to derive a matrix formulation for the plane elastostatic problem in terms of geometric density basis. This framework provides a promising foundation to resolve related problems, such as inverse problems of reconstructing elastic inclusions, neutral inclusions, and the effective property analysis of periodic structures. 

The convergence of the resulting series solution to the transmission problem and the interchangeability of summation orders, which depend on the decay of conformal mapping coefficients, remain important topics for future investigation.

\bibliography{2025_Elastic_lastname}{}
\bibliographystyle{plain}

%\bibliographystyle{plainnat}

% \ifx \bblindex \undefined \def \bblindex #1{} \fi\ifx \bbljournal \undefined
%   \def \bbljournal #1{{\em #1}\index{#1@{\em #1}}} \fi\ifx \bblnumber
%   \undefined \def \bblnumber #1{{\bf #1}} \fi\ifx \bblvolume \undefined \def
%   \bblvolume #1{{\bf #1}} \fi\ifx \noopsort \undefined \def \noopsort #1{} \fi
% \begin{thebibliography}{10}

% \end{thebibliography}

\end{document}